\documentclass[a4paper,10pt]{article}

\usepackage{amsmath,amssymb,amsthm,amscd}
\usepackage[mathscr]{eucal}
\usepackage[all]{xy}
\usepackage{color}
\usepackage{graphicx}

\makeatletter
    
    \@addtoreset{equation}{section}
  \makeatother

\newcommand{\plim}{\varprojlim}
\newcommand{\mcal}{\mathcal}
\newcommand{\mbf}{\mathbf}

\newcommand{\mfrak}{\mathfrak}
\newcommand{\mbb}{\mathbb}
\newcommand{\mrm}{\mathrm}
\newcommand{\vphi}{\varphi}

\newcommand{\aet}{\mathrm{\acute{e}t}}
\newcommand{\cO}{\mathcal{O}}

\newtheorem{theorem}{Theorem}[section]
\newtheorem{corollary}[theorem]{Corollary}
\newtheorem{lemma}[theorem]{Lemma}
\newtheorem{proposition}[theorem]{Proposition}

\theoremstyle{definition}
\newtheorem{definition}[theorem]{Definition}
\newtheorem{remark}[theorem]{Remark}

\newtheorem{example}[theorem]{Example}

\newtheorem*{acknowledgments}{Acknowledgments}

\title{Some Kummer extensions over maximal cyclotomic fields,  
a finiteness theorem of Ribet and TKND-AVKF fields}

\author{Takahiro Murotani\footnote{
Faculty of Arts and Sciences, Kyoto Institute of Technology, 
Matsugasaki, Sakyo-ku, Kyoto 606-8585, JAPAN
\endgraf
e-mail: {\tt murotani@kit.ac.jp}}
 and 
Yoshiyasu Ozeki\footnote{
Faculty of Science, Kanagawa University,
3-27-1 Rokkakubashi, Kanagawa-ku, Yokohama-shi, Kanagawa 221-8686, JAPAN
\endgraf
e-mail: {\tt ozeki@kanagawa-u.ac.jp}} 
}

\begin{document}
\maketitle

\begin{abstract}
It is a theorem of Ribet that an abelian variety defined over 
a number field $K$ has only finitely many torsion points
with values in the maximal cyclotomic extension field $K^{\mathrm{cyc}}$ of $K$.
Recently, R\"ossler and Szamuely generalized Ribet's theorem 
in terms of  the \'etale cohomology with $\mathbb{Q}/\mathbb{Z}$-coefficients of a smooth proper variety.
In this paper, we show that the same finiteness holds even after replacing $K^{\mathrm{cyc}}$
with the field obtained by adjoining to $K$ all roots of all elements of 
a certain subset of $K$.
Furthermore, we give some new examples of TKND-AVKF fields;
the notion of TKND-AVKF is introduced  by Hoshi, Mochizuki and Tsujimura,
and  TKND-AVKF fields are expected as one of suitable base fields 
for anabelian geometry.
\end{abstract}

      \tableofcontents

%%%%%%%%%%%%%%%%%%%%%%%%%%%%%%%%%%%%%%%%%%%%%%%%%%%%%%%%%%%%%%%%%%%%%%%%%%%%%%%%%%%%%%%%%%%%%%%%%%%%%%%%%%%
%%%%%%%%%%%%%%%%%%%%%%%%%%%%%%%%%%%%%%%%%%%%%%%%%%%%%%%%%%%%%%%%%%%%%%%%%%%%%%%%%%%%%%%%%%%%%%%%%%%%%%%%%%%
%                           1                              %%%%%%%%%%%%%%%%%%%%%%%%%%%%%%%%%%%%%%%%%%%%%%%%
%%%%%%%%%%%%%%%%%%%%%%%%%%%%%%%%%%%%%%%%%%%%%%%%%%%%%%%%%%%%%%%%%%%%%%%%%%%%%%%%%%%%%%%%%%%%%%%%%%%%%%%%%%%
%%%%%%%%%%%%%%%%%%%%%%%%%%%%%%%%%%%%%%%%%%%%%%%%%%%%%%%%%%%%%%%%%%%%%%%%%%%%%%%%%%%%%%%%%%%%%%%%%%%%%%%%%%%

\section{Introduction}

Let $K$ be a number field (= a finite extension field of the field of rational numbers $\mbb{Q}$).
The Mordell-Weil theorem asserts that the group $A(K)$ of $K$-rational points 
of an abelian variety $A$ over $K$ is finitely generated.
In particular, the torsion subgroup of $A(K)$ is finite.
We consider the  finiteness of  the $L$-rational torsion subgroup 
$A(L)_{\mrm{tor}}$ of $A$
for an algebraic extension $L$ of $K$ {\it of infinite degree}.
Motivated by Mazur's celebrated paper \cite{Ma}, 
Imai \cite{Im} and Serre \cite{Se} proved independently that, 
for any prime $p$, 
the group $A(L)_{\mrm{tor}}$ is finite
if $L$ is the cyclotomic $\mbb{Z}_p$-extension field of $K$.  
Moreover, Ribet showed in the appendix of \cite{KL} that 
the same finiteness holds also for  the case where $L$ is 
the maximal cyclotomic extension field $K^{\mrm{cyc}}$ of $K$.
If $L$ is the maximal abelian extension field $K^{\mrm{ab}}$ of $K$,
then it is a result of Zarhin \cite{Za} that $A(L)_{\mrm{tor}}$ is finite if and only if 
any non-zero $K$-simple subvariety of $A$ is not of CM-type over $K$.
(Here, we say that a $K$-simple abelian variety $X$ is {\it of CM-type over $K$}
if $\mrm{End}_K(X)\otimes_{\mbb{Z}} \mbb{Q}$ is a
number field of degree $2\dim X$.)
On the other hand, as an essentially non-abelian extension field case result,
Bogomolov showed  in \cite{Bo} that $A(L)_{\mrm{tor}}$ is finite if 
the intersection of $L$ and $K^{\mrm{ab}}$ has finite degree over $K$.
It is also an interesting observation by Zhang \cite{Zh} that 
 $A(L)_{\mrm{tor}}$ is finite if $L$ is the composite of 
$K$ and the maximal totally real subfield of $\overline{\mbb{Q}}$.

Let us explain our results. 
We denote by $\mfrak{Primes}$  the set of prime numbers. 
We denote by $G_F$ the absolute Galois group of a field $F\, (\subset \overline{\mbb{Q}})$.
For a subset $S$ of $F$, we denote by $S^{1/\infty}\subset \overline{\mbb{Q}}$
the set of all roots of all elements of $S$
and, for any prime $p$, denote by $S^{1/p^{\infty}}\subset \overline{\mbb{Q}}$ 
the set of all $p$-power roots of all elements of $S$.
For any algebraic variety $X$ over a number field $K$,
we denote by $h_i(X)$ the $i$-th Betti number of the topological space $X(\mbb{C})$
(which is independent of the choice of 
an embedding $\overline{K}\hookrightarrow \mbb{C}$),
and also set $X_{\overline{K}}:=X\otimes_K \overline{K}$.
Recently, R\"ossler and Szamuely \cite{RS} generalized Ribet's theorem above 
in terms of  the \'etale cohomology with $\mbb{Q}/\mbb{Z}$-coefficients of a smooth proper variety.
Our first main theorem below is motivated by their works. 

\begin{theorem}[= A part of Theorem \ref{MT:tor2}]
\label{MT:tor}
Let $K$ be a number field
and $p_0$ the maximal prime ramified in 
the maximal abelian subextension $K_0$ in $K/\mbb{Q}$
{\rm (}we set $p_0:=1$ if $K_0=\mbb{Q}${\rm )}.
Let $h>0$ be an integer and 
$\Delta$ a finitely generated subgroup  of $K^{\times}$, 
and set 
$$
M:=K(\Delta^{1/p^{\infty}}, K^{1/{q}^{\infty}}
\mid p,q\in \mfrak{Primes}, p\le \mrm{max}\{h+1, p_0\}<q).
$$
Let $i$ be an odd integer, $j$ an integer and 
$X$ a smooth proper geometrically connected algebraic variety 
over $K$ with $h_i(X)\le h$. 
Then, the group $H^i_{\aet}(X_{\overline{K}},\mbb{Q}/\mbb{Z}(j))^{G_M}$ is finite.
\end{theorem}

Note that the twist $j$ above
does not really play a role in the statement since $G_M$ 
fixes all roots of unity.
Now we consider the case where $X=A$ is an abelian variety.
As is explained in the Introduction of \cite{RS}, 
the torsion subgroup $A(M)_{\mrm{tor}}$  of $A(M)$ is isomorphic to 
$H^1_{\aet}(A^{\vee}_{\overline{K}},\mbb{Q}/\mbb{Z}(1))^{G_{M}}$
where $A^{\vee}$ is the dual abelian variety of $A$.
Therefore, we have the following.
\begin{corollary}
\label{AVfinite}
Let the notations $K$, $p_0$ and $\Delta$ 
be as in Theorem \ref{MT:tor} and $g>0$ an integer. 
Set 
$$
M':=K(\Delta^{1/p^{\infty}}, K^{1/{q}^{\infty}}
\mid p,q\in \mfrak{Primes}, p\le \mrm{max}\{2g+1, p_0\}<q).
$$
Then, for any abelian variety $A$ over $K$ of dimension $\le g$, 
 the torsion subgroup $A(M')_{\mrm{tor}}$ 
of $A(M')$ is finite.
\end{corollary}

The fields $M$ and $M'$ above contain $K^{\mrm{cyc}}$ (moreover, they contain $K(\Delta^{1/\infty})$).
Hence our results are refinements of  Ribet \cite{KL} 
and R\"ossler-Szamuely \cite{RS}.
However, it should be remarked that our proofs rely on their results; 
we reduce proofs to the case 
of maximal cyclotomic extension fields.
\if0
For abelian varieties without complex multiplication, our arguments
and the main theorem of Zarhin \cite{Za} 
assure some refined version of Corollary \ref{AVfinite}; see Corollary \ref{MT:Zarhin}.
\fi

Here is an immediate consequence of our results: 
For a number field $K$  with integer ring $\cO_K$,
the torsion subgroup $A(K((\cO_K^{\times})^{1/\infty}))_{\mrm{tor}}$ 
of $A(K((\cO_K^{\times})^{1/\infty}))$ is finite for any abelian variety over $K$.
(In fact, $\Delta:=\cO_K^{\times}$ is finitely generated by Dirichlet's unit theorem,
and $K((\cO_K^{\times})^{1/\infty})$ is contained in $M'$ appeared in Corollary \ref{AVfinite} for every $g$.)
So it seems quite natural to ask the following question:
``Is the torsion subgroup $A(K(K^{1/\infty}))_{\mrm{tor}}$ 
of $A(K(K^{1/\infty}))$ finite 
for any abelian variety $A$ over $K$?"
Now we do not have an answer to this question,
but instead, we will show that the field $K(K^{1/\infty})$ 
satisfies an important property {\it ``TKND-AVKF"}.

The methods of the proof of our theorems give some contribution to anabelian geometry. 
Hoshi, Mochizuki and Tsujimura showed in \cite{HMT} 
(the semi-absolute version of) the anabelian Grothendieck conjecture
 for higher dimensional ($\geq 2$) configuration
spaces associated to hyperbolic curves of genus $0$ 
over {\it TKND-AVKF} fields contained in $\overline{\mathbb{Q}}$.
(The notion of TKND-AVKF is defined by some smallness or vanishingness  
for subgroups consisting of divisible elements of  certain 
multiplicative groups  or Mordell-Weil groups; see Definition \ref{anabelianDEF}.)
Moreover, in \cite{Tsu3}, Tsujimura also showed the (relative) birational version 
of the anabelian Grothendieck conjecture for smooth curves 
over TKND-AVKF fields of characteristic $0$ with a certain mild condition.
So TKND-AVKF fields should be considered as one of suitable base fields
for anabelian geometry.
However, the definition of a TKND-AVKF field
seems to be a little bit ``strange" (at least for the authors) 
and thus it may be not so easy 
to understand their characters.  
From such a viewpoint, 
it would be important
to study examples or properties of TKND-AVKF fields.
As typical examples, finitely generated fields over prime fields are TKND-AVKF;
these fields are in the original situation that Grothendieck considered.
In addition, it is an interesting observation of Tsujimura that 
every subfield of the maximal cyclotomic extension field $K^{\mrm{cyc}}$ 
of a number field $K$ is also TKND-AVKF
(cf.\ \cite[the proof of Theorem 3.1 and Remark 3.4.1]{Tsu1}). 
Tsujimura furthermore showed in \cite{Tsu2} that, if $K$ is a number field and $p$ is a prime, 
then the field $K(p^{1/\infty})$ is a TKND-AVKF field.
Here, we give new examples of TKND-AVKF fields,
and some of them give  refinements of above results. 
Recall that a finite field extension is {\it solvable} 
if the Galois group of its Galois closure is solvable.

\begin{theorem}
\label{MT:AVKF}
Let $K$ be any one of the followings:
\begin{itemize}
\item[{\rm (a)}] $K$ is the composite field of all solvable extensions of degree $\le d$ 
over a given number field. Here, $d$ is a positive integer.
\item[{\rm (b)}] $K$ is the composite field of all finite extensions of degree $\le 4$ 
over a given number field. 
\item[{\rm (c)}] $K$ is a finite extension of 
$\mbb{Q}(\mu_{p^{n_p}}\mid p\in \mfrak{Primes})$
where $(n_p)_{p\in \mfrak{Primes}}$ is a family of positive integers. 
\end{itemize}
Then, any subfield of  $K(K^{1/\infty})$ is TKND-AVKF.
In particular,  $K^{\mrm{cyc}}$ is TKND-AVKF.
\end{theorem}

We also give a criterion of 
a TKND-AVKF property for 
the maximal abelian extension field  of a number field.
%(Note that $K^{\mrm{ab}}$ is always TKND.)
Here, we say that a  number field is a {\it CM field} if
it is a totally imaginary quadratic extension of 
a totally real number field. 
\begin{theorem}[= Corollary of Theorem \ref{Kab:AVKF}]
\label{Kab:TKND-AVKF}
Let $K$ be a number field 
and $K^{\mrm{ab}}$ the maximal abelian extension  field of $K$.
Then $K^{\mrm{ab}}$ is TKND-AVKF if and only if 
$K$ does not contain a CM field.
\end{theorem}

\begin{acknowledgments}
The authors would like to express their sincere gratitude to 
Yuichiro Taguchi for useful discussions and comments on the arguments of Section 2.  
They thank also Shota Tsujimura for explaining us his results \cite{Tsu3}.
The first author is supported by JSPS KAKENHI Grant Numbers JP22KJ1291 and JP24K16890.
The second author is supported by JSPS KAKENHI Grant Number JP19K03433.
\end{acknowledgments}

\vspace{5mm}
\noindent
{\bf Notation:} 
For any perfect field $F$,  
we denote by $G_F$  the absolute Galois group of $F$
(for a given algebraic closure $\overline{F}$ of $F$).
We denote by $\mu_n(F)$ the set of $n$-th roots of unity in $F$, 
$\mu_{p^{\infty}}(F):=\bigcup_{m\ge 0} \mu_{p^m}(F)$
for any prime  $p$ and 
$\mu_{\infty}(F):=\bigcup_{n\ge 1} \mu_{n}(F)$. 
We denote by $F^{\mrm{cyc}}$ the maximal cyclotomic extension field 
$F(\mu_{\infty}(\overline{F}))$ of $F$,
and also denote by $F^{\mrm{ab}}$ 
the maximal abelian extension field of $F$.
For a subset $S$ of $F$, we denote by $S^{1/\infty}\subset \overline{F}$
the set of all roots of all elements of $S$
and, for any prime $p$, denote by $S^{1/p^{\infty}}\subset \overline{F}$ 
the set of all $p$-power roots of all elements of $S$.
It should be remarked that, if $S$ contains nonzero elements of $F$, then we have 
$\mu_{\infty}(\overline{F})\subset F(S^{1/\infty})$ and $\mu_{p^{\infty}}(\overline{F})\subset F(S^{1/p^{\infty}})$.

We denote by $\mfrak{Primes}$ the set of prime numbers.
For any prime $p$, 
we denote by $\chi_p$ (resp.\ $\bar{\chi}_p$) the $p$-adic cyclotomic character
(resp.\ the mod $p$ cyclotomic character) 
defined over an appropriate Galois group
in the context.

%\section{Proofs of main results}
\section{Finiteness and vanishing of cohomologies}
The goal in this section is to show Theorems \ref{MT:tor} and \ref{AVfinite}.
Before proofs, we give some preliminary results.
Especially, Proposition \ref{KT} is the most important key tool for our proofs,
which is essentially given by Kubo and Taguchi \cite{KT}.

\subsection{Some remarks on Kummer theory and Kubo-Taguchi's methods}

We start with some elementary properties.
\begin{proposition}
\label{pro-ell}
Let $p$ be a prime,
$F$ a perfect field of characteristic $\not= p$
and $S$ a subset of $F$.

\noindent
{\rm (1)}  Assume $F\supset \mu_{p^{\infty}}(\overline{F})$.
Then, $F(S^{1/{p^{\infty}}})$ is a pro-$p$ abelian extension of $F$.

\noindent
{\rm (2)}  Assume $F\supset \mu_{p^{\infty}}(\overline{F})$.
Let $\Delta$ be a finitely generated subgroup of $F^{\times}$
and  $\Delta_{\mrm{tor}}$ the torsion subgroup of $\Delta$.
Then, $\mrm{Gal}(F(\Delta^{1/p^{\infty}})/F)$ is isomorphic to 
$\mbb{Z}_{p}^d$ for some $d\le \mrm{rank}\ \Delta/\Delta_{\mrm{tor}}$.

\noindent
{\rm (3)} Let $F'$ be a Galois extension of $F$ with 
$F'\supset \mu_{p^{\infty}}(\overline{F})$.
Set $G:=\mrm{Gal}(F'(S^{1/p^{\infty}})/F)$
and $H:= \mrm{Gal}(F'(S^{1/p^{\infty}})/F')$.
Then, we have
$\sigma \tau \sigma^{-1}=\tau^{\chi_{p}(\sigma)}$
for $\sigma\in G$
and $\tau\in H$.
\end{proposition}

\begin{remark}
The extension $F'(S^{1/p^{\infty}})/F'$
in (3) is a pro-$p$ abelian extension by (1) so that 
there exists a natural $\mbb{Z}_{p}^\times$-action on 
its Galois group $H$. 
Thus the $\chi_{p}(\sigma)$-th power $\tau^{\chi_{p}(\sigma)}$ 
of $\tau$ in (3) is well-defined.
\end{remark}

\begin{proof}
Throughout the proof, we may assume that $S$ is not empty.

(1) First we note that the group $\mrm{Gal}(F(S^{1/p^{\infty}})/F)$
is naturally regarded as a closed subgroup of 
$\prod_{s\in S} \mrm{Gal}(F(s^{1/p^{\infty}})/F)$.
Furthermore, it follows from Kummer theory that 
$\mrm{Gal}(F(s^{1/p^{\infty}})/F)$ is isomorphic to 
$\{0\}$ or $\mbb{Z}_{p}$ if $s$ 
is an $p$-divisible element of $F^{\times}$ or not,
respectively. Now the result follows.

(2) Set $\delta:=\mrm{rank}\ \Delta/\Delta_{\mrm{tor}}$
and take lifts $x_1,\dots ,x_{\delta}\in \Delta$ of a basis of the free abelian group 
$\Delta/\Delta_{\mrm{tor}}$.  
It is not difficult to check that the field  
$F(\Delta^{1/p^{\infty}})$ is the composite of 
$F(x_i^{1/p^{\infty}})$ for all $1\le i\le \delta$.
Thus  the group $\mrm{Gal}(F(\Delta^{1/p^{\infty}})/F)$
is naturally regarded as a closed subgroup of 
$\prod^{\delta}_{i=1} \mrm{Gal}(F(x_i^{1/p^{\infty}})/F)$.
Since each $\mrm{Gal}(F(x_i^{1/p^{\infty}})/F)$ is  isomorphic to $\{0\}$ or $\mbb{Z}_{p}$,
we obtain the desired result.

(3) It suffices to show that, 
for $n>0$, $\sigma\in G$, $\tau\in H$ and any non-zero $x\in \overline{F}$
with $x^{p^n}\in S$, it holds that 
$\sigma \tau \sigma^{-1}(x)=\tau^{\chi_{p}(\sigma)}(x)$.
Let $\zeta\in \overline{F}$ be a primitive $p^n$-th root of unity.
Since $S$ is a subset of $F$, for any $\rho\in G$, 
we see that $\rho(x)=\zeta^{c(\rho)} x$ 
with some integer $c(\rho)$. 
We give some remarks for $c(\rho)$. 
First, we have $x=\rho \rho^{-1}(x)=\zeta^{\chi_{p}(\rho)c(\rho^{-1})+c(\rho)}x$
and thus $\zeta^{\chi_{p}(\rho)c(\rho^{-1})+c(\rho)}=1$.
Next, if $\rho\in H$, we see $\rho^n(x)=\zeta^{nc(\rho)}x$ for $n\in \mbb{Z}$,
and this equality holds also for $n\in \mbb{Z}_{p}$
by continuity. 
Form these remarks, for $\sigma\in G$ and $\tau\in H$, we find
$\sigma \tau \sigma^{-1}(x)=\sigma \tau (\zeta^{c(\sigma^{-1})} x)
=\sigma(\zeta^{c(\sigma^{-1})+c(\tau)}x)
=\zeta^{\chi_{p}(\sigma)(c(\sigma^{-1})+c(\tau))+c(\sigma)}x
=\zeta^{\chi_{p}(\sigma)c(\sigma^{-1})+c(\sigma)} 
\cdot \zeta^{\chi_{p}(\sigma)c(\tau)}x
=\tau^{\chi_{p}(\sigma)}(x)$
as desired.
\end{proof}

We give a slight refinement of  Kubo and Taguchi \cite[Lemma 2.2]{KT},
which plays an important role in the next section. 
Here we recall that, 
for a field $E$ and an  $E$-representation $V$ of a group $G$,
we say that $G$ acts unipotently on $V$ 
if all the eigenvalues of the $\sigma$-action on $V$ 
are $1$ for any $\sigma\in V$.

\begin{proposition}
\label{KT}
Let $G$ be a profinite group and $H$
a closed normal abelian subgroup of $G$.
Assume the following hypothesis:
\begin{itemize}
\item[{\rm ({\bf H})}]
There exist an integer $c>1$ and  $\sigma_0\in G$
such that, for any $\tau\in H$, it holds that $\sigma_0 \tau \sigma_0^{-1}=\tau^c$.
\end{itemize}
Let $d$ be a positive integer.
Let $E$ be a topological field and $V$ a continuous $E$-linear
representation of $G$ of dimension $\le d$.
We set 
$$
m:=\mrm{lcm}\{c^r-1\mid 1\le r\le d \}
$$
where $c$ is $($any choice of$)$ the constant in {\rm ({\bf H})}. Then the following hold.

\noindent
{\rm (1)} $H^m$ acts unipotently on $V$.

\noindent
{\rm (2)} There exists an open normal subgroup $H'$ of $H$
such that $H'$ acts unipotently on $V$.

\noindent
{\rm (3)} Assume that $H$ is pro-$p$ for some prime  $p$.
If we denote by $m(p)$ the $p$-part of $m$, then
$H^{m(p)}$ acts unipotently on $V$. 
\end{proposition}

\begin{proof}
Replacing $V$ with $V\oplus E^{d-\dim V}$, we may assume $\dim V=d$.
Take any $\tau \in H$ and let $\lambda_1,\dots ,\lambda_d$ 
be the eigenvalues of the $\tau$-action on $V$.
By ({\bf H}), we have an equality 
$\{\lambda_1,\dots ,\lambda_d\}=\{\lambda_1^c,\dots ,\lambda_d^c\}$
as multisets of $d$-elements.
 In other words, the $c$-th power map gives a permutation on 
the multiset $\{\lambda_1,\dots ,\lambda_d\}$.
Hence, for any $1\le i\le d$, there exists an integer $1\le r\le d$ such that 
$\lambda_i^{c^r}=\lambda_i$.
By definition of $m$, we have $\lambda_i^m=1$ for any $i$.
This shows (1). 
If furthermore $H$ is pro-$p$, 
the fact that  $\lambda_i$ is a root of unity 
implies that  $\lambda_i$  has a $p$-power order.
Thus we obtain  $\lambda_i^{m(p)}=1$ for any $i$. This shows (3).
It suffices to show (2). 
Since $H$ is abelian, if we take a finite extension $E'$ of $E$ large enough, 
the semisimplification of the restriction of $V\otimes_E E'$ to $H$ is a direct sum of 
characters $H\to E'^{\times}$. 
It follows from (1) that each character has values in the set of $m$-th roots of unity. 
In particular, each character factors through a finite quotient of $H$. 
This shows (2).  
\end{proof}

We often use the following  ``trace version" of the Brauer-Nesbitt theorem.

\begin{theorem}[{\cite[30.16]{CR} or \cite[Theorem 2.4.6 and Remark 2.4.7 (iv)]{Wi}}]
\label{BN}
Let $G$ be a group and $E$ a field.
Let $V$ and $V'$ be finite dimensional semi-simple $E$-representations of $G$.
Assume that $\mrm{char} \,E=0$ or 
$\mrm{char}\, E> \mrm{Max}\{\dim_E V, \dim_E V'\}$.
If $\mrm{Tr}(\sigma \mid V) = \mrm{Tr}(\sigma \mid V')$ for any $\sigma\in G$,
then we have $V\simeq V'$.
\end{theorem}

\begin{corollary}
\label{BN:unip}
Let $G$ be a group, $E$ a field and  
$V$ a finite dimensional non-zero $E$-representation of $G$.
Assume that $\mrm{char} \,E=0$ or 
$\mrm{char}\, E> \dim_E V$, 
and also assume  that $G$ acts unipotently on $V$.
Then,  any  irreducible  $G$-stable subquotient of $V$ is of dimension one with trivial $G$-action.
In particular, we have $V^G\not=0$.
\end{corollary}
\begin{proof}
Set $d:=\dim_E\, V$ and denote by $V^{\mrm{ss}}$ the semi-simplification of $V$.
Since we have $\mrm{Tr}(\sigma \mid V^{\mrm{ss}}) = \mrm{Tr}(\sigma \mid E^{d})\, (= d)$ 
for any $\sigma\in G$,
we obtain $V^{\mrm{ss}}\simeq E^{d}$ by Theorem \ref{BN}. 
Since any  irreducible  $G$-stable subquotient of $V$ is isomorphic to a composition factor of  $V^{\mrm{ss}}$,
the result follows.
\end{proof}

\subsection{Proof of Theorem \ref{MT:tor}}

\begin{proposition}
\label{Thm.A}
Let $p$ be a prime.
Let $K$ be a field of characteristic $0$ such that 
$\chi_{p}(G_K)$ is open in $\mbb{Z}_p^{\times}$.
Let $L$ be a Galois extension of $K$ with $L\supset \mu_{\infty}(\overline{K})$,
and set $M:=L(K^{1/\infty})$.
Let $E$ be an algebraic extension of $\mbb{Q}_p$
and $V$ a finite dimensional continuous $E$-representation
of $G_K$ which satisfies the following property; 
for each finite extension $L'$ of $L$, it holds that $V^{G_{L'}}=0$.
Then, we have $V^{G_{M}}=0$. 
\end{proposition}

\begin{proof}
By continuity of $G_K$-action on $V$, 
there exist a finite extension $E_0/\mbb{Q}_p$  contained in $E$
and a  continuous $E_0$-representation $V_0$ 
of $G_K$ such that $V\simeq V_0\otimes_{E_0} E$.
Thus we may assume that $E$ is a finite extension of $\mbb{Q}_p$.
Assume that $V^{G_{M}}$ is not zero.
We set $L_q:=L(K^{1/q^{\infty}})$ for any prime $q$.
It follows from Proposition \ref{pro-ell} (1) and  $L\supset \mu_{q^{\infty}}(\overline{K})$
that $L_q$ is a pro-$q$ extension of $L$.
In particular, for any prime $q\not=p$, we have $L_q\cap L_p=L$.
We also set $H'_p:=\mrm{Gal}(M/L_p)$ 
and $G:=\mrm{Gal}(M/K)$.
Since $M$ is the composite of all $L_q$ for all primes $q$, 
we have 
$H_p'\hookrightarrow 
\prod_{q\not=p} \mrm{Gal}(L_qL_p/L_p)
\simeq 
\prod_{q\not=p} \mrm{Gal}(L_q/L)$,
which in particular shows that 
$H'_p$ is pro-prime-to-$p$.
Now we consider a natural $G$-action on $V^{G_{M}}$
coming from the $G_K$-action on $V^{G_{M}}$.
Let $\cO$ be the integer ring of $E$.
By continuity of a Galois action, there exists an $\cO$-lattice $\mcal{L}$ in 
 $V^{G_{M}}$ which is stable under the $G$-action. 
Let $\rho\colon G\to GL_{\cO}(\mcal{L})
\simeq GL_{t}(\cO)$ be the continuous homomorphism obtained by 
the $G$-action on $\mcal{L}$
where $t=\dim_E V^{G_M}$.
Set $p':=p$ or $p':=4$ if $p\not=2$ or $p=2$,
respectively.
Let $U$ be the kernel of 
the restriction to $H'_p$  of 
the composite of $\rho$ and the projection 
$GL_{t}(\cO)\to GL_{t}(\cO/p'\cO)$.
Then $U$ is an open subgroup of $H'_p$ 
and $\rho(U)$ 
has values in the kernel of the projection 
$GL_{t}(\cO)\to GL_{t}(\cO/p'\cO)$.
Since this kernel is pro-$p$ but $H'_p$ is pro-prime-to-$p$, 
we find that $\rho(U)$ is trivial.
If we denote by $\tilde{L}_p$  the  finite subextension
in $M/L_p$ corresponding to $U$, we obtain $V^{G_{M}}=V^{G_{\tilde{L}_p}}$. 
Since $L$ is a Galois extension of $K$, we can 
take a finite extension $L'$ of $L$ so that 
$L'$ is a Galois extension of $K$ and $\tilde{L}_p\subset L'_p$
where $L'_p:=L'(K^{1/p^{\infty}})$. 
We consider $V^{G_{L'_p}}(\not=0)$ as a representation of $\mrm{Gal}(L'_p/K)$.
Take $\sigma_0\in \mrm{Gal}(L_p'/K)$ such that 
$c:=\chi_{p}(\sigma_0)$ is an integer $>1$.
(Such $\sigma_0$ exists since $\chi_p(G_K)$ is open in $\mbb{Z}_p^{\times}$.)
By Proposition \ref{pro-ell} (3), we have 
$\sigma_0\tau\sigma_0^{-1}=\tau^{c}$ for any $\tau\in \mrm{Gal}(L_p'/L')$.
Thus it follows from Proposition \ref{KT} (2) that 
there exists a finite extension $L''$ of $L'$ contained in $L'_p$
such that $\mrm{Gal}(L'_p/L'')$ acts unipotently on $V^{G_{L'_p}}$.
Thus $V^{G_{L''}}=(V^{G_{L'_p}})^{\mrm{Gal}(L'_p/L'')}$ is not zero by 
Corollary \ref{BN:unip}
but this contradicts the assumption on $V$.
\end{proof}

\begin{proposition}
\label{Thm.B}
Let $K$ be a field of characteristic $0$ such that 
$\chi_{p}(G_K)$ is open in $\mbb{Z}_{p}^{\times}$ for every prime $p$.
Let $L$ be a Galois extension of $K$ with $L\supset \mu_{\infty}(\overline{K})$.
Let $\{ \mbb{E}_{\ell} \}_{\ell\in \mfrak{Primes}}$ be a family of 
algebraic extensions $\mbb{E}_{\ell}$ of $\mbb{F}_{\ell}$,
and let $\mathbf{W}=\{ W_{\ell} \}_{\ell\in \mfrak{Primes}}$ be a family of 
finite dimensional continuous $\mbb{E}_{\ell}$-representations $W_{\ell}$
of $G_K$ satisfying the following property.
\begin{itemize} 
\item[{\rm ({\bf H1})}] For each finite extension $L'$ of $L$, it holds that $W_{\ell}^{G_{L'}}=0$ 
for all but finitely many primes  $\ell$.
\end{itemize}
Let $\mu, C>0$ be integers.
Assume that the following conditions hold:
\begin{itemize} 
\item[{\rm ({\bf H2})}] $\bar{\chi}_{p}(G_K)\supset (\mbb{F}_{p}^{\times})^{\mu}$ 
for every prime $p>\mu C$.
\item[{\rm ({\bf H3})}] $\dim_{\mbb{E}_\ell} W_{\ell} <C$
 for all but finitely many primes $\ell$.
\end{itemize}
Let $\Delta$ be a finitely generated subgroup of $K^{\times}$.
We set 
$$
M:=L(\Delta^{1/p^{\infty}}, K^{1/{q}^{\infty}}
\mid p,q\in \mfrak{Primes}, p\le \mu C<q).
$$ 
Then, we have $W_{\ell}^{G_M}=0$ for all but finitely many primes $\ell$. 
\end{proposition}

\begin{remark}
\label{Thm.B:Rem}
(1) Consider the case where $K$ is a number field.
We denote by $p_0$ the maximal prime 
ramified in the maximal abelian extension $K_0$ of $\mbb{Q}$
contained in $K$ (we set $p_0:=1$ if $K=\mbb{Q}$).
Then, we have  $\bar{\chi}_{p}(G_K)=\mbb{F}_{p}^{\times}$
for every prime $p>p_0$.
(In fact, $\bar{\chi}_{p}(G_K)$ is equal to $\mbb{F}_{p}^{\times}$ 
if and only if $K_0\cap \mbb{Q}(\mu_p)=\mbb{Q}$,
and this is satisfied if $p$ is unramified in $K_0$.)

\noindent
(2) Let  $K$ be a subfield of $\overline{\mbb{Q}}$
and $K_0/\mbb{Q}$ the maximal abelian subextension of $K/\mbb{Q}$.
Assume that the following condition holds:
\begin{itemize}
\item[-] There exist integers  $\mu,\nu>0$ such that 
the absolute ramification index of any finite place of $K_0$
above any prime $p>\nu$
is a divisor of $\mu$.
\end{itemize} 
Then, it follows $\bar{\chi}_{p}(G_K)\supset (\mbb{F}_{p}^{\times})^{\mu}$
for every prime  $p>\nu$.
(In fact, the quotient $\mbb{F}_p^{\times}/\bar{\chi}_p(G_K)$ is isomorphic to 
$\mrm{Gal}(\mbb{Q}(\mu_p)\cap K_0/\mbb{Q})$ and 
the  order  of this Galois group is a divisor of $\mu$
if $p>\nu$.) 
\end{remark}

\begin{proof}[Proof of Proposition \ref{Thm.B}]
By continuity of $G_K$-action on $W_{\ell}$, 
we may assume that $\mbb{E}_{\ell}$ is a finite field for each $\ell$.
Denote by $\Delta_{\mrm{tor}}$ the torsion subgroup of $\Delta$.
Throughout the proof, we may assume  $\Delta \not= \Delta_{\mrm{tor}}$.
First we define some notations:
Denote by $\delta>0$ the rank of the free abelian group $\Delta/\Delta_{\mrm{tor}}$.
For any prime $p$, we set 
\begin{align*}
L_{p}
:=
\left\{
\begin{array}{cl}
L(\Delta^{1/p^{\infty}}) &\quad 
\mbox{if $p\le\mu C$ }, \cr
L(K^{1/p^{\infty}})
&\quad 
\mbox{if $p>\mu C$}.
\end{array}
\right.
\end{align*}
It follows from Proposition \ref{pro-ell} (1) 
and  $L\supset \mu_{p^{\infty}}(\overline{K})$
that 
$L_p$ is a pro-$p$ extension of $L$.
Note that the field $M$ is the composite of all $L_{p}$
where $p$ ranges all primes.  
Denote by $p_i$ the $i$-th prime, that is, $p_1=2<p_2=3<p_3=5<\cdots $.
For $i\ge 1$,
let us denote by $L_{(i)}$ the composite 
of all $L_{p_j}$ for  all $1\le j\le i$,
and set $L_{(0)}:=L$ for convenience.
Summary, we have 
$$
L=L_{(0)}\subset L_{(1)}\subset \cdots \subset 
L_{(i)}=L_{p_1}L_{p_2}\cdots L_{p_i} \subset 
\cdots \subset  M=\bigcup^{\infty}_{j=1} L_{(j)}.
$$
Since $L$ is a Galois extension of $K$, we see that 
$L_p$, $L_{(i)}$ and $M$ are also Galois extensions of $K$.

With above notations, let us start the proof of the proposition.
Let $\ell$ be any prime such that   
$\ell\ge C>\dim_{\mbb{E}_{\ell}} W_{\ell}$
and  $W_{\ell}^{G_M}$ is not zero.
The goal is to show that such $\ell$ must be  bounded above by 
some constant depending only on $\mathbf{W}, K,L, \Delta,\mu $ and $C$.
Let $s$ be any integer which satisfies the  properties that 
$p_s>\mu C$ and $p_s$ is strictly larger than the maximal prime divisor of 
the order of $GL_{C}(\mbb{E}_{\ell})$.
The Galois group $\mrm{Gal}(M/K)$ acts on $W_{\ell}^{G_M}$.
Since $\mrm{Gal}(M/L_{(s)})$ is prime-to-$p_i$ for any $1\le i\le s$
and $GL_{\mbb{E}_{\ell}}(W_{\ell}^{G_M})$ is (non-canonically) isomorphic 
to a subgroup of $GL_{C}(\mbb{E}_{\ell})$
by $C>\dim_{\mbb{E}_{\ell}} W_{\ell}$,
we find that $\mrm{Gal}(M/L_{(s)})$ acts trivially on $W_{\ell}^{G_M}$.
Thus we have $W_{\ell}^{G_M}=W_{\ell}^{G_{L_{(s)}}}$,
which is   not zero and is equipped with a natural  $\mrm{Gal}(L_{(s)}/K)$-action.
Now we define an integer $t \, (< s)$ by $p_t\le \mu C<p_{t+1}$; 
$$
p_1=2<p_2=3<\cdots <p_t\le \mu C
<p_{t+1}<\cdots <p_s.
$$

{\bf STEP 1.} We show that $W_{\ell}^{G_{L_{(t)}}}$ is not zero. 
Suppose that $p_i>\mu C$ and $W_{\ell}^{G_{L_{(i)}}}$ is not zero.
(Of course this holds for $i=s$.) 
Set $G:=\mrm{Gal}(L_{(i)}/K)$ and $H:=\mrm{Gal}(L_{(i)}/L_{(i-1)})$.
The group $G$ acts on $W_{\ell}^{G_{L_{(i)}}}$.
Since $\chi_{p_i}(G)$ is open in $\mbb{Z}^{\times}_{p_i}$
and $\bar{\chi}_{p_i}(G)\supset (\mbb{F}_{p_i}^{\times})^{\mu}$ by $p_i>\mu C$ 
and ({\bf H2}),
there exists an integer $c_i>1$ such that 
$c_i \mod p_i$ generates the cyclic group $(\mbb{F}^{\times}_{p_i})^{\mu}$
and $c_i=\chi_{p_i}(\sigma_i)$ for some $\sigma_i\in G$.
In particular, we have $\sigma_i \tau \sigma_i^{-1}=\tau^{c_i}$
for $\tau\in H$ by Proposition \ref{pro-ell} (3).  
Now we note that $c_i^r-1$ is prime to $p_i$ for $1\le r <C$.
(In fact, if we denote by $n_i$ the order of $c_i \mod p_i$,
then we have 
$C\le \frac{p_i-1}{\mu}\le \frac{p_i-1}{\mrm{gcd}\{\mu,p_i-1\}} = n_i$.)
Since $H$ is a closed normal subgroup of $G$
and also is pro-$p_i$ abelian by Proposition \ref{pro-ell} (2), 
it follows from Proposition \ref{KT} (3) that 
$H$ acts unipotently on $W_{\ell}^{G_{L_{(i)}}}$.
This gives the fact that 
$W_{\ell}^{G_{L_{(i-1)}}}=(W_{\ell}^{G_{L_{(i)}}})^H$ is not zero by
$\ell>\dim_{\mbb{E}_\ell} W_{\ell}$ and 
Corollary \ref{BN:unip}.
Repeating above arguments from the case $i=s$ to the case $i=t+1$,
we obtain the desired result.

\vspace{2mm}

{\bf STEP 2.} Assume that $W(i):=W_{\ell}^{G_{L^{i}_{(i)}}}$ is not zero
for some $1\le i\le t$ and some finite extension $K^{i}$ of $K$,
where  $L^{i}_{(i)}:=K^{i} L_{(i)}$.
(By STEP 1, we are in this situation if we set $i=t$ and $K^{t}=K$.)
Under this assumption, we show that 
$W(i-1):=W_{\ell}^{G_{L^{i-1}_{(i-1)}}}$ is not zero
by choosing a certain finite extension $K^{i-1}$ of $K^{i}$
depending only on $i,K^{i},L,\Delta$ and $C$.
Here, $L^{i-1}_{(i-1)}:=K^{i-1} L_{(i-1)}$.

(In the rest of the proof, if a notation $X$ is determined 
only by some given elements $\ast_1,\ast_2,\dots $,
then we often write $X=X\langle \ast_1,\ast_2,\dots \rangle$.)
Set $G_i:=\mrm{Gal}(L^{i}_{(i)}/K^{i})$ and 
$H_i:=\mrm{Gal}(L^{i}_{(i)}/L^{i}_{(i-1)})$
where $L^{i}_{(i-1)}:=K^iL_{(i-1)}$.
The group $G_i$ acts on $W(i)$.
Since $H_i$ is abelian, 
the semi-simplification of the restriction of
$W(i)\otimes_{\mbb{E}_\ell} \overline{\mbb{F}}_\ell$ to $H_i$
is isomorphic to 
$\overline{\mbb{F}}_\ell(\psi_1)\oplus \overline{\mbb{F}}_\ell(\psi_2) \oplus \cdots $
for some continuous characters $\psi_i\colon H_i\to \overline{\mbb{F}}^{\times}_\ell$.
Fix a choice of an integer 
$c=c\langle p_i,K^{i}\rangle =c\langle i,K^{i}\rangle>1$ 
such that $c\in \chi_{p_{i}}(G_{K^{i}})$.
(Such $c$ exists since $\chi_{p_{i}}(G_{K^{i}})$ is open in $\mbb{Z}_{p_i}^{\times}$.)
If we denote by $m=m\langle p_i,c,C\rangle
=m\langle i,K^{i},C\rangle$ the $p_i$-part of 
$
\mrm{lcm}\{c^r-1\mid 1\le r< C \},
$
it follows from Proposition \ref{pro-ell} (3) and Proposition \ref{KT} (3) that 
$H_i^{m}$ acts unipotently on $W(i)$.
This implies that  each $\psi_j$ has values in 
the set  of $m$-th roots of unity in $\overline{\mbb{F}}^{\times}_\ell$.
Since $H_i$ is isomorphic to $\mbb{Z}_{p_i}^{\delta_i}$ for some $\delta_i\le \delta$
by Proposition \ref{pro-ell} (2),
there exist only finitely many open subgroups of $H_i$
of  index $\le m$. 
Hence, the intersection  $H_i'$ of all such open subgroups of $H_i$ 
is of finite index in $H_i$. 
By construction, each $\psi_j$ is trivial on $H_i'$ 
and thus we find that $H_i'$ acts unipotently on $W(i)$.
In particular,  $W(i)^{H_i'}$ is not zero by 
$\ell>\dim_{\mbb{E}_\ell} W_{\ell}$
and Corollary \ref{BN:unip}.
Let us denote by  $N/L^{i}_{(i-1)}$ the finite subextension in 
$L^{i}_{(i)}/L^{i}_{(i-1)}$ corresponding to $H_i'$.
Since $L^{i}_{(i)}=K^{i}L(\Delta^{1/p_j^{\infty}}\mid 1\le j\le i)$
and $L^{i}_{(i-1)}=K^{i}L(\Delta^{1/p_j^{\infty}}\mid 1\le j\le i-1)$,
we find that $N$ is determined depending only 
on $K^{i}$, $L$, $p_1,\dots ,p_i$, $\Delta$ and $m$, namely, 
$N=N\langle K^{i}, L, p_1,\dots ,p_i, \Delta, m \rangle
=N\langle i,K^{i},L, \Delta,C \rangle $.
By construction of $N$, we know that 
$W_{\ell}^{G_{N}} \, (=W(i)^{H_i'} )$ is not zero.
Choose a finite extension $K^{i-1}$ of $K^{i}$
so that  $N= K^{i-1} L^{i}_{(i-1)}=L^{i-1}_{(i-1)}$.
Then   $W_{\ell}^{G_{L^{i-1}_{(i-1)}}}$ is not zero.
Since we may choose $K^{i-1}$ depending only on  
$i,K^{i},L, \Delta$ and $C$, 
we obtain the desired result.

\vspace{2mm}

{\bf STEP 3.} 
By STEP 1 and STEP 2, 
we obtain a sequence of finite extensions
$$
K=K^{t}\subset K^{t-1}\subset K^{t-2}\subset \cdots \subset K^1 \subset K^0
$$ 
(in $M$) depending only on $t, K, L, \Delta$ and $C$,
with the property that 
$W_{\ell}^{G_{L^i_{(i)}}}$ is not zero for any $0\le i\le t$.
Here, $L^{i}_{(i)}:=K^{i} L_{(i)}$.
In particular, $W_{\ell}^{G_{L'}}$ is not zero where  $L':=L^0_{(0)}=K^0L$.
We should remark that, since $t$ is uniquely determined by the condition 
$p_t\le \mu C<p_{t+1}$, 
one may say that the above field extensions are determined by 
$K, L, \Delta,\mu$ and $C$.  
By ({\bf H1}), there exists  a constant $B=B\langle \mathbf{W}, L'\rangle$ 
such that $\ell<B$.
Since the construction of $L'$ depends only on $K, L, \Delta,\mu$ and $C$,
we obtain the desired condition $\ell<B=B\langle \mathbf{W}, K,L, \Delta, \mu, C\rangle$.
\end{proof}

For a subfield $K$ of $\overline{\mbb{Q}}$, we say that 
{\it $K$ is of uniformly bounded ramification indexes} if 
there exists an integer $\nu>0$ such that the 
absolute ramification index of any finite place of $K$
is at most $\nu$.
Clearly,  any number field is of uniformly bounded ramification  indexes.
Moreover, if we denote by $\mbb{Q}(d)_{\mrm{ab}}$ 
the composite of all abelian extensions of $\mbb{Q}$ of degree $\le d$ for a given $d$,
we find that $\mbb{Q}(d)_{\mrm{ab}}$  is of uniformly bounded ramification indexes
(see Proposition \ref{unif} below).
We recall that $h_i(X)$ is the $i$-th Betti number of the topological space $X(\mbb{C})$
for any algebraic variety $X$ over $K$
(which is independent of the choice of 
an embedding $\overline{K}\hookrightarrow \mbb{C}$).
It is well-known that $h_i(X)$ 
coincides with the $\mbb{Q}_{\ell}$-dimension of the \'etale cohomology
$H^i_{\aet}(X_{\overline{K}},\mbb{Q}_{\ell})$ 
since there exists a (natural) isomorphism 
$H^i_{\aet}(X_{\overline{K}},\mbb{Q}_{\ell})\simeq 
H^i_{\mrm{sing}}(X(\mbb{C}),\mbb{Q})\otimes_{\mbb{Q}} \mbb{Q}_{\ell}$
where $H^i_{\mrm{sing}}(X(\mbb{C}),\mbb{Q})$ is the singular cohomology of $X(\mbb{C})$.
We also recall that $X_{\overline{K}}:=X\otimes_K \overline{K}$.

Now we are ready  to prove  Theorem \ref{MT:tor}.  
In fact, we can prove the following refined statement.

\begin{theorem}
\label{MT:tor2}
Let $K$ be a subfield of a finite extension of $\mbb{Q}^{\mrm{cyc}}$
with uniformly bounded ramification  indexes. 
Let $X$ be a smooth proper geometrically connected algebraic variety over $K$.
Let $i>0$ be an odd integer.

\noindent
{\rm (1)} There exists a constant $D>0$, depending only on 
$K$ and $h_i(X)$, which satisfies the following property:
Let $\Delta$ be a finitely generated subgroup  of $K^{\times}$.
We set 
$$
M:=K(\Delta^{1/p^{\infty}}, K^{1/{q}^{\infty}}
\mid p,q\in \mfrak{Primes}, p\le D<q).
$$
Then, the group $H^i_{\aet}(X_{\overline{K}},\mbb{Q}/\mbb{Z}(j))^{G_M}$ is finite 
for any $j$. 

\noindent
{\rm (2)} If we set  $D$ as follows, then $D$ satisfies 
the property in {\rm (1)}. 
Here, $K_0$ is the maximal abelian subextension 
in $K/\mbb{Q}$.
\begin{itemize}
\item[{\rm (2-1)}] $D=\mu (h_i(X)+1)$. 
Here,  $\mu>0$ is any integer such that 
the absolute ramification index of any finite place of $K_0$ 
above any prime $p>\mu (h_i(X)+1)$ is a divisor of $\mu$. 
{\rm (}Such $\mu$ exists 
since $K$ is of uniformly bounded ramification  indexes{\rm )}.
\item[{\rm (2-2)}] $D=\mrm{max}\{h_i(X)+1, p_0\}$
if $K$ is a number field. 
Here,  $p_0$ is the maximal prime 
ramified in $K_0$  {\rm (}we set $p_0:=1$ if $K_0=\mbb{Q}${\rm )}.
\end{itemize}

\end{theorem}

\begin{proof} 
First we claim that there exists a constant 
$\ell_0=\ell_0(X,i)$  depending only on $X$ and $i$ 
with the property that $\dim_{\mbb{F}_{\ell}} W_{\ell}=h_i(X)$
for any prime  $\ell>\ell_0$
where  $W_{\ell}:=H^i_{\aet}(X_{\overline{K}},\mbb{F}_\ell(j))$.
This should be well-known to experts but we include an explanation
for the sake of completeness.
For any prime   $\ell$, 
the group $W_{\ell}$ is finite and 
the $\mbb{Z}_\ell$-module 
$H^i_{\aet}(X_{\overline{K}},\mbb{Z}_{\ell}(j))$ is finitely generated.
Furthermore, there exists a constant $\ell_0=\ell_0(X,i)$ 
depending only on $X$ and $i$ 
that the torsion subgroup 
of $H^i_{\aet}(X_{\overline{K}},\mbb{Z}_{\ell}(j))$ 
and $H^{i+1}_{\aet}(X_{\overline{K}},\mbb{Z}_{\ell}(j))$ are trivial 
for any prime $\ell>\ell_0$ by a result of Gabber \cite{Ga}. 
The $\mbb{Z}_\ell$-rank of $H^i_{\aet}(X_{\overline{K}},\mbb{Z}_{\ell}(j))$ 
coincides with the $\mbb{Q}_\ell$-dimension of 
$H^i_{\aet}(X_{\overline{K}},\mbb{Q}_{\ell}(j))$, 
which is just the $i$-th Betti number $h^i(X)$ of 
$X(\mbb{C})$. 
Note that we have an exact sequence 
$
0\to H^i_{\aet}(X_{\overline{K}},\mbb{Z}_{\ell}(j))/\ell 
\to W_{\ell}
\to H^{i+1}_{\aet}(X_{\overline{K}},\mbb{Z}_\ell(j))[\ell]
\to 0
$
of $G_K$-modules
coming from the long exact sequence associated with the multiplication-by-$\ell$ map.
Thus, for $\ell>\ell_0$, we have 
$\dim_{\mbb{F}_\ell} W_{\ell}
=\dim_{\mbb{F}_\ell} H^i_{\aet}(X_{\overline{K}},\mbb{Z}_{\ell}(j))/\ell 
=\mrm{rank}_{\mbb{Z}_\ell}\, H^i_{\aet}(X_{\overline{K}},\mbb{Z}_{\ell}(j))
=h^i(X)$.
Hence the claim follows.

Proofs for the cases (2-1) and (2-2) proceed in parallel.
We define some notations in each cases. Let $\mu>0$ and $C>0$ be as follows.
\begin{itemize}
\item Under the consideration of (2-1), let $\mu$ be as in (2-1) and $C:=h_i(X)+1$.
\item Under the consideration of (2-2), 
set $\mu:=1$ and $C:=\mrm{max}\{h_i(X)+1,p_0\}$.
\end{itemize}
In both cases, we have $D=\mu C$.
Set $L:=K^{\mrm{cyc}}$. 
Let  $M$ be the field as in the theorem, that is, 
$M=L(\Delta^{1/p^{\infty}}, K^{1/{q}^{\infty}}
\mid p,q\in \mfrak{Primes}, p\le D <q)$.

We proceed the proof of the theorem with above notations.
By Proposition 2.1 of \cite{RS},
it suffices to show 
\begin{itemize}
\item[(A)] $H^i_{\aet}(X_{\overline{K}},\mbb{Q}_\ell(j))^{G_M}=0$ 
for all primes $\ell$, and 
\item[(B)] $H^i_{\aet}(X_{\overline{K}},\mbb{F}_\ell(j))^{G_M}=0$ for all but finitely many 
primes $\ell$.
\end{itemize}
Here we remark that 
$\chi_p(G_K)$ is open in $\mbb{Z}_p^{\times}$
for every prime $p$ since $K$ is of 
uniformly bounded ramification  indexes. 
Set $V:=H^i_{\aet}(X_{\overline{K}},\mbb{Q}_{\ell}(j))$  and 
$M_{\infty}:=K(K^{1/\infty})$.
Since $L$ is a finite extension of $\mbb{Q}^{\mrm{cyc}}$, 
it follows from Proposition 3.1 of \cite{RS} that 
$V^{G_{L'}}=0$ for any finite extension $L'/L$.
Thus we obtain $V^{G_{M_{\infty}}}=0$  by Proposition \ref{Thm.A},
which shows (A) since $M$ is a subfield of   $M_{\infty}$. 
For the proof of (B), it is enough to check that 
the conditions  ({\bf H1}),  ({\bf H2}) and  ({\bf H3}) appeared 
in Proposition \ref{Thm.B} are satisfied 
under our situation with additional notation $\mbf{W}:=\{W_{\ell}\}_{\ell\in \mfrak{Primes}}=
\{H^i_{\aet}(X_{\overline{K}},\mbb{F}_\ell(j))\}_{\ell\in \mfrak{Primes}}$.
However, it is not difficult to check them; 
({\bf H1}) is a consequence of Proposition 3.4 of \cite{RS},
({\bf H2}) follows from Remark  \ref{Thm.B:Rem}  
and  ({\bf H3}) is clear by   the claim above.
\end{proof}

Here is an immediate consequence of 
Theorem \ref{MT:tor2}.
\begin{corollary}
\label{Thm.B2}
Let $K$ be a subfield of a finite extension of $\mbb{Q}^{\mrm{cyc}}$
with uniformly bounded ramification  indexes. 
Let $\{\Delta_{p}\}_{p\in \mfrak{Primes}}$ be a family of 
finitely generated subgroups $\Delta_{p}$ of $K^{\times}$.
We set 
$$
M:=K(\Delta_{p}^{1/p^{\infty}}\mid p\in \mfrak{Primes}).
$$
Let $X$ be a smooth proper geometrically connected algebraic variety over $K$.
Then, the group $H^i_{\aet}(X_{\overline{K}},\mbb{Q}/\mbb{Z}(j))^{G_M}$ is finite
for any odd $i$ and any $j$.
\end{corollary}

The result below is an immediate consequence of Proposition  \ref{Thm.A}
(see the argument of (A) in the proof of Theorem   \ref{MT:tor2}).
\begin{theorem}
\label{MT:ptor}
%Let $K$ be 
% a subfield of a finite extension of 
%$\mbb{Q}(\mu_{{p}^{n_{p}}}\mid p\in \mfrak{Primes})$
%where $(n_{p})_{p\in \mfrak{Primes}}$ is a family of positive integers.
Let $K$ be a subfield of a finite extension of $\mbb{Q}^{\mrm{cyc}}$
with uniformly bounded ramification  indexes. 
We set 
$$
M:=K(K^{1/\infty}).
$$
Let $X$ be a smooth proper geometrically connected algebraic variety over $K$.
Then, we have $H^i_{\aet}(X_{\overline{K}},\mbb{Q}_{\ell}(j))^{G_{M}}=0$  
for any odd $i$, any $j$ and any prime  $\ell$.
\end{theorem}

\begin{remark}
The condition that ``$K$ is  of
 uniformly bounded ramification  indexes"
cannot be simply removed from the statement of Theorem \ref{MT:ptor}.

For example, we consider the case where $K=\mbb{Q}(\mu_{p^{\infty}}\mid p\in \mfrak{Primes})=\mbb{Q}^{\mrm{cyc}}$
and set $M:=K(K^{1/\infty})\, (=\mbb{Q}((\mbb{Q}^{\mrm{cyc}})^{1/\infty}))$.
Put $K_0=\mbb{Q}(\sqrt{-1})$, which is a subfield of $K$.
The elliptic curve $E\colon y^2=x^3+x$ has $j$-invariant $1728$ and 
$\mrm{End}_{\mbb{C}}(E)$ is isomorphic to the integer ring $\mbb{Z}[\sqrt{-1}]$ of $K_0$.
As is explained in \cite[Example 5.8]{Si},
we know that $K_0^{\mrm{ab}}$ coincides with the field  $K_0(E_{\mrm{tor}})$
generated by $K_0$ and all the torsion points of $E$. 
On the other hand, since any finite abelian extension of $K_0^{\mrm{cyc}} (=\mbb{Q}^{\mrm{cyc}})$ 
is contained in $K_0^{\mrm{cyc}}((K_0^{\mrm{cyc}})^{1/n})$ for some $n>1$ by Kummer theory,
we know that $K_0^{\mrm{ab}}\subset K_0^{\mrm{cyc}}((K_0^{\mrm{cyc}})^{1/\infty})=M$.
Therefore, we find that $E(M)[\ell^{\infty}]$ is infinite for any prime $\ell$.
This is equivalent to say that 
$H^1_{\aet}(E_{\overline{K}},\mbb{Q}_{\ell}(1))^{G_{M}}$, 
which is isomorphic to the $G_{M}$-fixed part of 
the $\ell$-adic rational Tate module $V_{\ell}(E)$ of $E$, is not zero
for any prime $\ell$.
\end{remark}

As a final topic in this section,
we study abelian extensions of $\mbb{Q}$ which are of uniformly bounded ramification indexes.
For any integer $d>0$,
we denote by $\mbb{Q}(d)_{\mrm{ab}}$ 
the composite of all abelian extensions of $\mbb{Q}$ of degree $\le d$.
In addition, for any prime $p$, we denote by $\mbb{Q}_p(d)_{\mrm{ab}}$ 
the composite field  of all abelian extensions of $\mbb{Q}_p$ of degree $\le d$.

\begin{proposition}
\label{unif}
{\rm (1)}  $\mbb{Q}(d)_{\mrm{ab}}$ is of uniformly bounded ramification indexes.

\noindent
{\rm (2)} 
If $K$ is an abelian extension of $\mbb{Q}$,
then $K$ is of uniformly bounded ramification indexes
if and only if $K$ is a subfield of $\mbb{Q}(d)_{\mrm{ab}}$  for some $d$.
\end{proposition}

\begin{proof}
(1) 
It suffices to show that
there exists a constant $\delta_d$ which depends only on $d$ 
such that, for any prime $p$, we have  
$[\mbb{Q}_p(d)_{\mrm{ab}}:\mbb{Q}_p]\le \delta_d$
(In fact, for any embedding $\iota \colon \overline{\mbb{Q}}\hookrightarrow \overline{\mbb{Q}}_p$,
the composite field of $\iota(\mbb{Q}(d)_{\mrm{ab}})$ and $\mbb{Q}_p$ is contained in $\mbb{Q}_p(d)_{\mrm{ab}}$.
Remark that this also implies that the absolute ramification index of any finite place of $\mbb{Q}(d)_{\mrm{ab}}$ above $p$
is a divisor of that of $\mbb{Q}_p(d)_{\mrm{ab}}$). 
By local class field theory, 
there exists a bijection between 
the set of abelian extensions of $\mbb{Q}_p$ of degree $\le d$ 
and the set of index $\le d$ subgroups of $\mbb{Q}^{\times}_p/(\mbb{Q}^{\times}_p)^{d!}$.
It is easy to check that there exists a constant $c_d$, depending only on $d$,
such that the cardinality of the latter set is less than $c_d$. 
Therefore, setting $\delta_d:=d^{c_d}$,
we obtain the desired result.

(2) By (1),  it suffices to show ``only if" part.
Let $d$ be the least common multiple of 
the absolute ramification indexes of all finite places of $K$, 
which is finite by assumption on $K$.
The goal is to show that $K$ is a subfield of $\mbb{Q}(d)_{\mrm{ab}}$.
Take any $\alpha\in K$.
Denote by $S$ the set of rational primes which are ramified in $\mbb{Q}(\alpha)$
and  also denote by $\mbb{Q}_S$ the maximal abelian extension of $\mbb{Q}$
unramified outside $S$.
We have a natural isomorphism 
$$
\mrm{Gal}(\mbb{Q}_S/\mbb{Q})\overset{\sim}\longrightarrow
\prod_{p\in S} \mbb{Z}_p^{\times}
$$
via cyclotomic characters;
 we identify $\mrm{Gal}(\mbb{Q}_S/\mbb{Q})$ with 
$\prod_{p\in S} \mbb{Z}_p^{\times}$.
Let $F_p$ be the maximal unramified subextension of $\mbb{Q}_S/\mbb{Q}$ 
at $p\in S$ and put $I_p=\mrm{Gal}(\mbb{Q}_S/F_p)$. 
Note that $I_p=\mbb{Z}_p^{\times}$
and these are the inertia subgroup of $\mrm{Gal}(\mbb{Q}_S/\mbb{Q})$ at $p$. 
Put $I_p'=\mrm{Gal}(\mbb{Q}_S/F_p(\alpha))$ for each $p\in S$.
We define the subextension $M$  of $\mbb{Q}_S/\mbb{Q}$ 
by $\mrm{Gal}(\mbb{Q}_S/M)=\prod_{p\in S} I_p'$,
and also define $M_{q}$ the subextension of $\mbb{Q}_S/\mbb{Q}$ 
by $\mrm{Gal}(\mbb{Q}_S/M_q)=I_q'\times\prod_{p\in S, p\not=q} I_p$
for any $q\in S$.
Since $\mrm{Gal}(M_q/\mbb{Q})$ is isomorphic to $I_q/I_q'$, 
we see $[M_q:\mbb{Q}]=[F_q(\alpha):F_q]$ and these are equal to the absolute ramification index of 
$\mbb{Q}(\alpha)$ at $q\in S$, which is a divisor of $d$.
Since $M$ is the composite of all $M_q$ for $q\in S$, 
we have $\mbb{Q}(\alpha)\subset M\subset \mbb{Q}(d)_{\mrm{ab}}$.
Therefore, we obtain $K\subset \mbb{Q}(d)_{\mrm{ab}}$ as desired.
\end{proof}

\begin{example}
\label{comp}
Let $K=\mbb{Q}(d)_{\mrm{ab}}$.
It follows from the Kronecker-Weber theorem that 
$K$ is a subfield of $\mbb{Q}^{\mrm{cyc}}$, and hence 
$K$ satisfies the required conditions in Theorem  \ref{MT:tor2}. 
If we denote by $e_p(d)$ the absolute ramification index of $\mbb{Q}_p(d)_{\mrm{ab}}$ for any prime $p$, 
then the absolute ramification index of any finite place of $K$ above $p$
is a divisor of $e_p(d)$. 
Hence one can verify that a constant $\mu$ in Theorem  \ref{MT:tor2} (2-1) 
can be chosen as 
$$
\mu=e(d):=\mrm{lcm}\{e_p(d)\}_{p\in \mfrak{Primes},p\not=2}.
$$
The constant $e(d)$ is finite since $K$ is of uniformly bounded ramification indexes,
and $e(d)=1$ if and only if $d=1$.
For example, if $d=2$ (thus $K=\mbb{Q}(\sqrt{m}\mid m\in \mbb{Z})$), 
then we have  $e(2)=2$ since $\mbb{Q}_p(2)_{\mrm{ab}}=\mbb{Q}_{p^2}(\sqrt{p})$ 
for any odd prime $p$. Here, $\mbb{Q}_{p^2}$ is the quadratic unramified extension of $\mbb{Q}_p$.
\end{example}

\if0
We say that 
an abelian variety $A$ over a number field $K$ has complex multiplication over $\overline{K}$
if $\mrm{End}_{\overline{K}}(A)\otimes_{\mbb{Z}}\mbb{Q}$ contains a number field of degree $2\dim A$.
\begin{corollary}
\label{MT:Zarhin}
Let $g>0$ be an integer, $K$ a number field and 
$\Delta$ a finitely generated subgroup of $K^{\times}$.
We denote by $p_0$ the maximal prime ramified 
in the maximal abelian extension $K_0$ of 
$\mbb{Q}$ contained in $K$ (we set $p_0:=1$ if $K_0=\mbb{Q}$).
We set 
$$
M:=K^{\mrm{ab}}(\Delta^{1/p^{\infty}}, K^{1/{q}^{\infty}}
\mid p,q\in \mfrak{Primes}, p\le \mrm{max}\{2g+1, p_0\}<q).
$$ 
Let $A$ be an abelian variety over $K$ of dimension $\le g$
without complex multiplication over $\overline{K}$. 
If $A$ is  $\overline{K}$-simple, then the torsion subgroup $A(M)_{\mrm{tor}}$ 
of $A(M)$ is finite.
\end{corollary}

\begin{proof}
(Essentially the same proof as those of Theorems \ref{MT:tor} or \ref{AVfinite} proceeds; 
we use   the result of \cite{Za} instead of the result of Ribet \cite{KL}.)
It suffices to prove that 
$A(M)[\ell^{\infty}]$ is finite
for every prime $\ell$
and  $A(M)[\ell]$ is zero
for all but finitely many primes $\ell$. 
The former follows from the main theorem of \cite{Za} and  Proposition \ref{Thm.A},
and the latter follows from the main theorem of {\it loc.\ cit.}, 
Proposition \ref{Thm.B} and Remark \ref{Thm.B:Rem}. 
\end{proof}
\fi

\section{TKND-AVKF fields}

In this section, we study TKND-AVKF property for some fields
and we show Theorems  \ref{MT:AVKF} and \ref{Kab:TKND-AVKF}
in the Introduction.
Following \cite[Definition 3.3]{Tsu1} and \cite[Definition 6.6]{HMT},
we first recall some notions.

\begin{definition}
\label{anabelianDEF}
Let $F$ be a field of characteristic $0$ and $p$ a prime.
We denote by $F^{\times p^{\infty}}$ and $F^{\times \infty}$
the set of $p$-divisible elements of $F^\times$ and
the set of divisible elements of $F^\times$, respectively;
$F^{\times p^{\infty}}=\bigcap_{n\ge 1} (F^{\times})^{p^n}$
and 
$F^{\times \infty}=\bigcap_{n\ge 1} (F^{\times})^{n}$.
We denote by $F_{\mrm{div}}$ the field obtained by adjoining to $\mbb{Q}$
the divisible elements of the multiplicative groups of 
finite extension fields of $F$, that is, 
$F_{\mrm{div}}:=\bigcup_{E} \mbb{Q}(E^{\times \infty})$
where $E(\subset\overline{F})$ ranges all finite extensions of $F$.

\noindent
(1) We say that $F$ is {\it TKND} (=``{\it torally Kummer-nondegenerate}") 
if the extension $\overline{F}/F_{\mrm{div}}$ is of infinite degree.

\noindent
(2) We say that $F$ is {\it AVKF} (=``{\it abelian variety Kummer-faithful}") if,
for every finite extension $E$ of $F$ and every abelian variety 
$A$ over $E$, the set of divisible elements of $A(E)$ is zero, that is, 
$$
\bigcap_{n\ge 1} nA(E)=0.
$$ 

\noindent
(3) We say that $F$ is {\it TKND-AVKF} if $F$ is both TKND and AVKF.

\noindent
(4) We say that $F$ is  {\it $\times \mu$-indivisible}
(resp.\ {\it $p$-$\times \mu$-indivisible}) 
if  $F^{\times \infty}\subset \mu_{\infty}(F)$
(resp.\  $F^{\times p^{\infty}}\subset \mu_{\infty}(F)$).
We say that $F$ is {\it stably $\times \mu$-indivisible} (resp.\ {\it stably $p$-$\times \mu$-indivisible})
if any finite extension field of $F$  is $\times \mu$-indivisible (resp.\ $p$-$\times \mu$-indivisible).

\noindent
(5) We say that $F$ is  {\it $\mu_{p^{\infty}}$-finite}  
if  $\mu_{p^{\infty}}(F)$ is finite.
We say that $F$ is {\it stably $\mu_{p^{\infty}}$-finite} 
if any finite extension field of $F$  is $\mu_{p^{\infty}}$-finite.
\end{definition}

\begin{remark}
By considering the Weil restrictions of abelian varieties over a field $F$,
one verifies immediately that 
$F$ is AVKF 
if and only if the set of divisible elements of $A(F)$ is zero 
for every abelian variety $A$ over $F$. 
\end{remark}

For various properties of the above notions, 
it will be helpful to the readers to refer \cite[\S 6]{HMT}.
It should be remarked that the following implications hold:
\[
\xymatrix{
\mbox{stably $p$-$\times \mu$-indivisible} \ar@{=>}[r] \ar@{=>}[d] 
& \mbox{stably $\times \mu$-indivisible}  \ar@{=>}[r] \ar@{=>}[d] 
& \mbox{TKND}   
\\ 
\mbox{$p$-$\times \mu$-indivisible} \ar@{=>}[r] 
& \mbox{$\times \mu$-indivisible}  
& 
}
\]

%\subsection{Non-$\times \mu$-indivisible TKND-AVKF fields}
\subsection{A criterion of TKND property}

We study some criterion on TKND property.

\begin{proposition}
\label{TKND}
Let $K$ be a field of characteristic $0$ 
containing $\mu_\infty(\overline{K})$
and $L$  a Galois extension of $K$.
Assume that $K$ is stably $p$-$\times \mu$-indivisible
for some prime $p$.

\noindent
{\rm (1)} We have $L_{\mrm{div}}\subset L$.

\noindent
{\rm (2)} Let $M$ be a subfield of $L$ such that $L$ is algebraic over $M$.
If $L\not=\overline{K}$, then $M$ is TKND.
\end{proposition}

\begin{proof}
Admitting (1),  we show (2). 
We have $M_{\mrm{div}}\subset L_{\mrm{div}}\subset L
\subset \overline{K}=\overline{M}$. 
By $L\not=\overline{K}$ and the Artin-Schreier theorem,
we obtain that the extension $\overline{K}/L$ is of infinite degree.
Hence the extension $\overline{M}/M_{\mrm{div}}$ is also of infinite degree, 
which implies that $M$ is TKND.
Thus it suffices to show (1). 
Let $E$ be a finite extension of $L$, and $\alpha\in E^{\times\infty}$.
It suffices to show that $\alpha\in L$.
By replacing $E$ by the Galois closure of $E/K$ 
(which is finite over $L$), we assume that $E$ is Galois over $K$.
Let $K_1$ be the Galois closure of $K(\alpha)/K$.
Fix a prime  $p$  such that $K$ is stably $p$-$\times \mu$-indivisible. 
In this proof, we set $\mu_{\infty}:=\mu_{\infty}(\overline{K})$ to simplify notation.

It follows from the assumption on $K$ that $K_1$ is $p$-$\times\mu$-indivisible.
Therefore, if $\alpha\in K_1^{\times p^\infty}$, then $\alpha\in\mu_\infty\subset L$.
In the following, we assume that $\alpha\not\in K_1^{\times p^\infty}$.
Set $M_1:=K_1(\alpha^{1/p^{\infty}})$ (so, $\mathrm{Gal}(M_1/K_1)\simeq\mathbb{Z}_p$).
Since $\mathrm{Gal}(E/K_1L)$ is finite, we have $M_1\subset K_1L$.
Let $M_2$ be the field corresponding to $M_1$ 
under the natural isomorphism 
$\mathrm{Gal}(K_1L/K_1)\simeq\mathrm{Gal}(L/K_1\cap L)$, 
and set $K_2:=K_1\cap L$.
Then we have the following commutative diagram:
\[
\xymatrix{
0 \ar[r] & H^1(\mathrm{Gal}(M_1/K_1), \mathbb{Z}_p)  \ar[r] 
& H^1(G_{K_1},\,\mathbb{Z}_p)=:(\hat{K_1^\times})_p \ar[r] 
& H^1(G_{M_1}, \mathbb{Z}_p)=:(\hat{M_1^\times})_p  \\ 0 \ar[r] 
&  H^1(\mathrm{Gal}(M_2/K_2), \mathbb{Z}_p) \ar[r] \ar@{^{(}-_>}^\simeq[u] 
& H^1(G_{K_2},\,\mathbb{Z}_p)=:(\hat{K_2^\times})_p \ar@{^{(}-_>}[u] \ar[r] 
& H^1(G_{M_2},\,\mathbb{Z}_p)=:(\hat{M_2^\times})_p, \ar@{^{(}-_>}[u]}
\]
where the horizontal sequences are exact.
Note that the middle (resp. right) vertical arrow 
is injective since $K_1$ (resp. $M_1$) is finite over $K_2$ (resp. $M_2$).
Note also that $\mu_\infty\subset K_1$, 
and hence $(\hat{K_1^\times})_p\simeq\displaystyle\varprojlim_{n}K_1^\times/(K_1^\times)^{p^n}$, and so on.
Let $f_\alpha$ be the (nontrivial) element of $(\hat{K_1^\times})_p$ determined by $\alpha$.
Since the image of $f_\alpha$ in $(\hat{M_1^\times})_p$ is $1$, it holds that $f_\alpha\in(\hat{K_2^\times})_p\subset(\hat{K_1^\times})_p$.

On the other hand, since $K_1$ and $K_2$ are $p$-$\times\mu$-indivisible, 
we have the following commutative diagram:
\[\xymatrix{ 
0 \ar[r] & \mu_\infty \ar@{=}[d] \ar[r] 
& K_2^\times \ar[r] \ar@{^{(}-_>}[d] 
& (\hat{K_2^\times})_p \ar@{^{(}-_>}[d]  \\ 0 \ar[r] 
& \mu_\infty \ar[r] & K_1^\times \ar[r] 
& (\hat{K_1^\times})_p,}\]
where the horizontal sequences are exact.
Therefore, we have the following inclusions of $\mathrm{Gal}(K_1/K_2)$-modules:
\[\xymatrix{
K_2^\times/\mu_\infty \ar@{^{(}-_>}[d] \ar@{^{(}-_>}[r] 
& (\hat{K_2^\times})_p \ar@{^{(}-_>}[d] \\ K_1^\times/\mu_\infty \ar@{^{(}-_>}[r] 
& (\hat{K_1^\times})_p.
}
\]
Moreover, by considering the long exact sequence associated to
$
0 \to  \mu_\infty \to K_1^{\times} \to K_1^{\times}/\mu_{\infty} \to  0,
$
we obtain
$$
0 \to K_2^\times/\mu_\infty \to (K_1^\times/\mu_\infty)^{\mathrm{Gal}(K_1/K_2)} 
\to  H^1(\mathrm{Gal}(K_1/K_2), \mu_\infty)
=\mathrm{Hom}(\mathrm{Gal}(K_1/K_2), \mathbb{Q}/\mathbb{Z}).
$$
%Since  $p$ does not divide $[K_1:K_2]$, it holds that $(\hat{K_1^\times})_p^{\mathrm{Gal}(K_1/K_2)}=(\hat{K_2^\times})_p$.
Now, $f_\alpha\in K_1^\times/\mu_\infty\subset(\hat{K_1^\times})_p$ 
since the natural morphism $K_1^\times\to (\hat{K_1^\times})_p$ 
factors through $K_1^\times/\mu_\infty$, 
and $f_\alpha\in(\hat{K_2^\times})_p\subset(\hat{K_1^\times})_p^{\mathrm{Gal}(K_2/K_1)}$.
So, $f_\alpha\in(K_1^\times/\mu_\infty)^{\mathrm{Gal}(K_1/K_2)}$, and hence $f_\alpha^d\in K_2^\times/\mu_\infty$, where $d:=[K_1:K_2]$.
This shows that there exists $\zeta\in\mu_\infty$ satisfying $\zeta\alpha^d\in K_2^\times$.
However, since $\mu_\infty\subset K_2$, 
$\alpha^d$ is an element of $K_2\subset L$.
Consider the following commutative diagram:
\[\xymatrix
{ 0 \ar[r] & L^{\times\infty} \ar@{^{(}-_>}[d] \ar[r] 
& L^\times \ar[r] \ar@{^{(}-_>}[d] 
& \hat{L^\times}:=\varprojlim_{n}L^\times/(L^\times)^n=H^1(G_L, \hat{\mathbb{Z}}) \ar@{^{(}-_>}[d]  \\ 0 \ar[r] 
& E^{\times\infty} \ar[r] & E^\times \ar[r]
& \hat{E^\times}:=\varprojlim_{n}E^\times/(E^\times)^n=H^1(G_E, \hat{\mathbb{Z}}),}
\]
where the horizontal sequences are exact.
(The injectivity of right vertical arrow follows from 
the  condition that $E/L$ is a finite extension and $\mu_{\infty}\subset L$.)
Since $\alpha\in E^{\times\infty}$, the image of $\alpha$, 
hence also the image of $\alpha^d$, in $\hat{E^\times}$ is $1$.
Therefore, we have $\alpha^d\in L^{\times\infty}$.
As $\mu_\infty\subset L$, this shows that $\alpha\in L$, as desired.
\end{proof}

We say that a field $k$ is {\it sub-$p$-adic} if 
it is a subfield of a finitely generated extension of $\mbb{Q}_p$.
More generally, 
we say that a field $k$ is {\it generalized sub-$p$-adic} if 
it is a subfield of a finitely generated extension  of $\hat{\mbb{Q}}^{\mrm{ur}}_p$,
where $\hat{\mbb{Q}}^{\mrm{ur}}_p$ is the completion of the maximal unramified extension field of $\mbb{Q}_p$.

\begin{example}\label{st-st}
(1) Let $k$ be a field of characteristic $0$ and $K:=k^{\mrm{cyc}}$.
Suppose that $k$ is both stably $\mu_{p^\infty}$-finite and stably $p$-$\times\mu$-indivisible for a prime $p$ (e.g., $k$ is generalized sub-$p$-adic (cf. Lemma D (iii) of \cite{Tsu1})).
Then, by Lemma D (iv) of \cite{Tsu1}, $K$ is stably $p$-$\times\mu$-indivisible.
In particular, $K$ is stably $p$-$\times\mu$-indivisible for every prime $p$ if $k$ is a number field.

(2)
Let $d>0$ be an integer and denote by $\mbb{Q}(d)$ 
the composite of all number fields of degree $\le d$.
Then, any subfield $k$ of $\mbb{Q}(d)$ is both stably $\mu_{p^{\infty}}$-finite and stably $p$-$\times \mu$-indivisible
for every prime $p$, and hence $K$ is stably $p$-$\times \mu$-indivisible
for every prime $p$.
In fact, for any prime $p$, if we fix an embedding 
$\overline{\mbb{Q}}\hookrightarrow \overline{\mbb{Q}}_p$,
then $k\mbb{Q}_p$ is a finite extension of $\mbb{Q}_p$.
In particular, $k$ is sub-$p$-adic.
\end{example}

%\begin{example}
%Let $d>0$ be an integer and denote by $\mbb{Q}(d)$ 
%the composite of all number fields of degree $\le d$.
%Then, any subfield $K$ of $\mbb{Q}(d)$ satisfies 
%the assumption in Proposition \ref{TKND}, that is, 
%$K$ is both stably $\mu_{p^{\infty}}$-finite and stably $p$-$\times \mu$-indivisible
%for every prime $p$.
%
%In fact, for any prime $p$, if we fix an embedding 
%$\overline{\mbb{Q}}\hookrightarrow \overline{\mbb{Q}}_p$,
%then $K\mbb{Q}_p$ is a finite extension of $\mbb{Q}_p$
%(since there exist only finitely many $p$-adic fields of a given degree) 
%and thus Lemma D (iii) of \cite{Tsu1} assures that  $K\mbb{Q}_p$ (hence also $K$) 
%is both stably $\mu_{p^{\infty}}$-finite and stably $p$-$\times \mu$-indivisible.
%\end{example}

\begin{example}
(1) Let $F$ be a number field.
If $G_F$ has a nontrivial finite subgroup $H$ (e.g., $F=\mathbb{Q}$), the fixed field $\overline{F}^H$ of $H$ is real closed (note that $H$ is necessarily of order $2$).
Therefore, as in \cite[Remark 1.1.3]{Tsu2}, $\overline{F}^H$ is not TKND.

\vspace{2mm}

(2) Let $K$ be a complete discrete valuation field whose residue characteristic is $0$.
Then $K$ is not TKND (cf. \cite[Remark 1.4.2]{Tsu3}).
Indeed, there exist a subfield $k$ of the ring of integers of $K$ which is isomorphic to the residue field of $K$ 
(via the natural surjection from the ring of integers to the residue field),
 and a uniformizer $t$ of $K$ such that $K=k(\!(t)\!)$.
Let $\overline{K}$ be an algebraic closure of $K$ and $\overline{k}$ the algebraic closure of $k$ in $\overline{K}$.
To verify the above assertion, it suffices to verify the following:
\begin{itemize}
\item[(a)] For any positive integer $n$ and any $n$-th root $t_n\in\overline{K}$ of $t$, $t_n$ belongs to $K_{\mrm{div}}$.

\item[(b)] For any $\alpha\in\overline{k}$, $\alpha$ belongs to $K_{\mrm{div}}$.
\end{itemize}
Assertion (a) follows immediately from the fact that $1+t_n$ is divisible in $K(t_n)^\times$.
Assertion (b) follows immediately from (a) and the fact that $1+\alpha t$ is divisible in $K(\alpha)^\times$.
\end{example}

\begin{example}\label{wild}
In Proposition \ref{TKND}, the condition that $L$ is Galois over $K$ is crucial.
Let $F$ be a number field, $p$ a prime and $\sigma\in G_F$ an element such that the closure $\overline{\langle\sigma\rangle}$ of the subgroup generated by $\sigma$ in $G_F$ is isomorphic to $\mathbb{Z}_p$.
(For example, let $\mathfrak{p}$ be a nonarchimedean prime of $F$ with residue characteristic $p$, and $F_{\mathfrak{p}}$ the field corresponding to the decomposition subgroup of $G_F$ associated to $\mathfrak{p}$ (determined up to conjugacy).
Then $F_{\mathfrak{p}}$ is isomorphic to the henselization of $F$ with respect to $\mathfrak{p}$ (see, e.g., \cite[\S 2.3, Proposition 11]{BLR}), and $G_{F_{\mathfrak{p}}}$ is identified with the absolute Galois group of the completion of $F$ at $\mathfrak{p}$.
Any nontrivial element $\sigma$ of the ``wild inertia subgroup'' of $G_{F_{\mathfrak{p}}}$ satisfies the above condition.)
We claim that the extension $L$ of $F$ corresponding to $\overline{\langle\sigma\rangle}$ satisfies $L\cdot L_{\mathrm{div}}=\overline{L}$, and, in particular, $L_{\mathrm{div}}\not\subset L$. % (i.e., $L$ does not satisfy the condition (1) of Proposition \ref{TKND}).
(Note that, this, together with Proposition \ref{TKND}, shows that there exist no algebraic extensions $K$ of $F$ satisfying the conditions given in Proposition \ref{TKND} over which $L$ is Galois.)

For any nonnegative integer $n$, we denote by $L_n$ the intermediate field of $\overline{L}/L$ such that $[L_n:L]=p^n$.
Note that, for any prime $\ell\neq p$ and any nonnegative integer $n$, $L_n^\times$ is $\ell$-divisible.
To verify the above claim, it suffices to verify the following:
\begin{quote}
For any nonnegative integer $n$, $L_{n+1}^{\times\infty}\setminus L_n^\times$ is not empty.
\end{quote}
Let us verify this assertion.
Let $n$ be a nonnegative integer.
Here, we claim that $L_{n+1}^\times/L_n^\times\cdot L_{n+1}^{\times\infty}$ is finite.
Indeed, by considering the long exact sequence associated to the short exact sequence $\xymatrix{1 \ar[r] & \mu_\infty \ar[r] & L_{n+1}^{\times\infty} \ar[r] & L_{n+1}^{\times\infty}/\mu_\infty \ar[r] & 1}$ of $\mathrm{Gal}(L_{n+1}/L_n)$-modules, we obtain the following exact sequence:
\[\xymatrix{H^1(\mathrm{Gal}(L_{n+1}/L_n),\,\mu_\infty) \ar[r] & H^1(\mathrm{Gal}(L_{n+1}/L_n),\,L_{n+1}^{\times\infty}) \ar[r] & H^1(\mathrm{Gal}(L_{n+1}/L_n),\,L_{n+1}^{\times\infty}/\mu_\infty).}\]
Since $H^1(\mathrm{Gal}(L_{n+1}/L_n),\,\mu_\infty)\,(\simeq\mathrm{Hom}(\mathrm{Gal}(L_{n+1}/L_n),\,\mathbb{Q}/\mathbb{Z}))$ is finite and $L_{n+1}^{\times\infty}/\mu_\infty$ is uniquely divisible (hence cohomologically trivial (cf. \cite[Proposition 1.6.2]{NSW})), $H^1(\mathrm{Gal}(L_{n+1}/L_n),\,L_{n+1}^{\times\infty})$ is also finite.
Moreover, by considering the long exact sequence associated to the short exact sequence $\xymatrix{1 \ar[r] & L_{n+1}^{\times\infty} \ar[r] & L_{n+1}^\times \ar[r] & L_{n+1}^\times/L_{n+1}^{\times\infty} \ar[r] & 1}$ of $\mathrm{Gal}(L_{n+1}/L_n)$-modules, we obtain the following exact sequence:
\[\xymatrix{1 \ar[r] & L_n^\times/L_n^{\times\infty} \ar[r] & (L_{n+1}^\times/L_{n+1}^{\times\infty})^{\mathrm{Gal}(L_{n+1}/L_n)} \ar[r] & H^1(\mathrm{Gal}(L_{n+1}/L_n),\,L_{n+1}^{\times\infty}) \ar[r] & 1.}\]
(Here, note that $(L_{n+1}^{\times\infty})^{\mathrm{Gal}(L_{n+1}/L_n)}=L_n^{\times\infty}$ since $L_{n+1}/L_n$ is finite.)
This shows that the subgroup $L_n^\times/L_n^{\times\infty}$ of $(L_{n+1}^\times/L_{n+1}^{\times\infty})^{\mathrm{Gal}(L_{n+1}/L_n)}$ is of finite index.

On the other hand, since $L_{n+1}$ contains all roots of unity, it holds that $H^1(G_{L_{n+1}},\,\mathbb{Z}_p)\simeq\displaystyle\varprojlim_m L_{n+1}^\times/(L_{n+1}^\times)^{p^m}$.
The kernel of the natural homomorphism $L_{n+1}^\times\to\displaystyle\varprojlim_m L_{n+1}^\times/(L_{n+1}^\times)^{p^m}$ is $L_{n+1}^{\times p^\infty}$, and this coincides with $L_{n+1}^{\times\infty}$.
Moreover, since $G_{L_n}$ is abelian, the action of $\mathrm{Gal}(L_{n+1}/L_n)$ on $H^1(G_{L_{n+1}},\,\mathbb{Z}_p)$ is trivial.
In particular, we have $(L_{n+1}^\times/L_{n+1}^{\times\infty})^{\mathrm{Gal}(L_{n+1}/L_n)}=L_{n+1}^\times/L_{n+1}^{\times\infty}$.
Therefore, $L_{n+1}^\times/L_n^\times\cdot L_{n+1}^{\times\infty}=(L_{n+1}^\times/L_{n+1}^{\times\infty})/(L_{n}^\times/L_{n}^{\times\infty})$ is finite.
This completes the proof of the above claim.

Now, let us consider the following exact sequence:
\[\xymatrix{1 \ar[r] & L_{n+1}^{\times\infty}/L_n^{\times\infty} \ar[r] & L_{n+1}^\times/L_n^\times \ar[r] & L_{n+1}^\times/L_n^\times\cdot L_{n+1}^{\times\infty} \ar[r] & 1.}\]
Since $L_{n+1}^\times/L_n^\times$ is infinite, $L_{n+1}^{\times\infty}/L_n^{\times\infty}\,(=L_{n+1}^{\times\infty}/(L_n^{\times}\cap L_{n+1}^{\times\infty}))$ is also infinite.
In particular, $L_{n+1}^{\times\infty}\setminus L_n^\times$ is not empty, as desired.

However, the authors at the time of writing do not know whether $L$ is TKND or not (cf. Remark \ref{KF}).
\end{example}

\begin{remark}\label{KF}
Let $F$ be a field of characteristic $0$.
We say that $F$ is {\it{Kummer-faithful}} (resp. {\it{torally Kummer-faithful}}) if, for every finite extension $F'$ of $F$ and every semi-abelian variety (resp. every torus) $A$ over $F'$, the set of divisible elements of $A(F')$ is zero, that is,
\[\bigcap_{n\geq 1}nA(F')=0.\]
Note that (torally) Kummer-faithful fields are TKND.
Let $K$ be a number field and $e$ a positive integer.
For $\sigma=(\sigma_1, \cdots , \sigma_e)\in G_K^e$, set $\overline{K}(\sigma)$ to be the fixed field of $\sigma$ in $\overline{K}$, and $\overline{K}[\sigma]$ to be the maximal Galois subextension of $K$ in $\overline{K}(\sigma)$.
It is known that, any finite extension of $\overline{K}[\sigma]$ is Kummer-faithful for almost all $\sigma\in G_K^e$ (in terms of the (normalized) Haar measure).
(See \cite[Corollary 1]{Oh1} and \cite{Oh2} for the case where $e\geq 2$, and \cite[Theorem 5.3]{AT} for the general case).
Moreover, if $e\geq 2$, any finite extension of $\overline{K}(\sigma)$ is Kummer-faithful for almost all $\sigma\in G_K^e$ (cf. \cite[Theorem 5.2]{AT}).

These results show that ``almost all'' of the algebraic extensions of number fields are TKND.
However, these results do not imply that the field $L$ constructed in Example \ref{wild} is TKND since (the wild inertia subgroups of) the decomposition subgroups of primes are measure zero sets.
\end{remark}

\subsection{TKND-AVKF property for maximal abelian extension fields of number fields}

Let $K$ be a number field and $K^{\mrm{ab}}$\, ($\subset \overline{\mbb{Q}}$) the maximal abelian extension  field of $K$.
It is shown in Proposition 1.2 of  \cite{Tsu2} 
that $K^{\mrm{ab}}$ is stably $\times \mu$-indivisible, 
and in particular, $K^{\mrm{ab}}$ is TKND.
On the other hand, we also know that 
$K^{\mrm{ab}}\,(= \mbb{Q}^{\mrm{cyc}}$) is AVKF if $K=\mbb{Q}$ but 
$K^{\mrm{ab}}$ is not AVKF if $K=\mbb{Q}(\sqrt{-1})$ (cf.\ \cite[Proposition 2.6]{Tsu2}).
Thus it depends on the choice of a number field $K$
whether $K^{\mrm{ab}}$ is AVKF or not.
We give an explicit criterion for this phenomena.
\begin{theorem}
\label{Kab:AVKF}
Let $K$ be a number field 
and $K^{\mrm{ab}}$ the maximal abelian extension  field of $K$.
Then the following are equivalent.
\begin{itemize}
\item[{\rm (1)}] $K$ does not contain a CM field.
\item[{\rm (2)}] For every abelian variety $A$ over $K$,
$A(K^{\mrm{ab}})_{\mrm{tor}}$ is finite.
\item[{\rm (3)}] $K^{\mrm{ab}}$ is AVKF.
\end{itemize}
\end{theorem}
The goal in this section is to show the above theorem.
Note that Theorem \ref{Kab:TKND-AVKF} is an immediate consequence of Theorem 
\ref{Kab:AVKF} since $K^{\mrm{ab}}$ is always TKND.

Before a proof of Theorem \ref{Kab:AVKF},
basically following \cite[\S 3, \S 5 and \S 8]{Sh2}, 
we recall some results related with {\it CM-type} of abelian variety with complex multiplication.
In the following, if $A$ is an abelian variety defined over a field $k$ of characteristic zero,
we denote by $\mrm{End}_{k}(A)$ the ring of $k$-endomorphisms of $A$
and set $\mrm{End}_{k}(A)^0:=\mrm{End}_{k}(A)\otimes_{\mbb{Z}} \mbb{Q}$. 
We say that $A$ is {\it $k$-simple} if it has no abelian $k$-subvariety other than $\{0\}$ and itself, 
and also we say that $A$ is {\it simple} if $A$ is $\bar{k}$-simple.

Let $F$ be a number field of degree $2g$ and $A$ a $g$-dimensional abelian variety over $\mbb{C}$. 
Suppose that an injective ring homomorphism 
$\theta\colon F\hookrightarrow \mrm{End}_{\mbb{C}}(A)^0$ is given.
The abelian variety $A$ is isomorphic to a $g$-dimensional complex torus;
we fix an isomorphism $A(\mbb{C})\simeq \mbb{C}^g/D$ where $D$ is a lattice in $\mbb{C}^g$. 
Let $\lambda$ be an element of $\mrm{End}_{\mbb{C}}(A)$.
The map $\lambda$  corresponds to a linear map 
$\Lambda\colon \mbb{C}^g\to \mbb{C}^g$ such that $\Lambda(D)\subset D$.
With respect to a given coordinate-system of $\mbb{C}^g$,
$\Lambda$ is represented by a matrix $S\in M_g(\mbb{C})$. 
The mapping $\lambda\mapsto S$ can be uniquely extended to a 
ring homomorhism $\mrm{End}_{\mbb{C}}(A)^0\to M_g(\mbb{C})$,
which is called the {\it analytic representation} of $\mrm{End}_{\mbb{C}}(A)^0$. 
We abuse notation by writing $S$ for the analytic representation.
Furthremore, 
since $\Lambda$ as above preserves $D$, 
$\Lambda$ restricted to $D$ defines a map $D\to D$ and 
it is represented by   
a matrix $M\in M_{2g}(\mbb{Z})$ 
with respect to a given basis of $D$. 
The mapping $\lambda\mapsto M$ can be uniquely extended to a 
ring homomorhism $\mrm{End}_{\mbb{C}}(A)^0\to M_{2g}(\mbb{Q})$,
which is called the {\it rational representation} of $\mrm{End}_{\mbb{C}}(A)^0$. 
We abuse notation by writing $M$ for the rational representation.
By definition, one can check that $M$ (as a $\mbb{C}$-representation) 
is equivalent to the direct sum of $S$ and its complex conjugate $\bar{S}$.
Now let us denote by $\vphi_1,\dots ,\vphi_{2g}$ all the $\mbb{Q}$-algebra embeddings 
from $F$ into $\mbb{C}$. 
It follows from \cite[\S 5.1, Lemma 1]{Sh2} that 
the representation $M$ restricted to $F$ is equivalent to 
the direct sum of $\vphi_1,\dots ,\vphi_{2g}$.
Hence, by reordering subscripts, we find that 
$S$ restricted to $F$ is equivalent to 
the direct sum of $\vphi_1,\dots ,\vphi_{g}$, and 
$\bar{S}$ restricted to $F$ is equivalent to 
the direct sum of $\vphi_{g+1},\dots ,\vphi_{2g}$,
which is the direct sum of the complex conjugates $\bar{\vphi}_1,\dots ,\bar{\vphi}_{g}$
of $\vphi_1,\dots ,\vphi_{g}$.
Setting $\Phi:=\{\vphi_1,\dots ,\vphi_{g}\}$, 
we say that $(A,\theta)$ is {\it  of type} $(F;\Phi)$.
If we denote by $\Gamma_F=\{\vphi_1,\dots ,\vphi_{2g}\}$
the set of all   $\mbb{Q}$-algebra embeddings 
from $F$ into $\mbb{C}$ and 
set $\bar\Phi:=\{ \vphi_{g+1},\dots ,\vphi_{2g}\}=\{\bar\vphi_1,\dots ,\bar\vphi_{g}\}$,
then we observe that we have $\Gamma_F=\Phi\bigcup \bar\Phi$
and $\Phi\bigcap \bar\Phi=\emptyset$.
\begin{definition}
\label{CM-type}
Let $F$ be a number field of degree $2g$ and $\Phi$ 
a set of $g$-distinct $\mbb{Q}$-algebra embeddings from $F$ into $\mbb{C}$.
We say that $(F;\Phi)$ is a {\it CM-type} if there exist a
$g$-dimensional abelian variety $A$ over $\mbb{C}$ and 
an injective ring homomorphism $\theta\colon F\hookrightarrow \mrm{End}_{\mbb{C}}(A)^0$ 
such that $(A,\theta)$ is of type $(F;\Phi)$. 
\end{definition}
Here we remark that, it is a theorem of Oort \cite{Oo}
that $A$ as in the above definition 
always has a model over some number field. 
Thus we can choose $A$ as an abelian variety defined over $\overline{\mbb{Q}}$.
(With this choice of $A$, note that we have $\mrm{End}_{\mbb{C}}(A)^0=\mrm{End}_{\overline{\mbb{Q}}}(A)^0$.)

Let $F$ be a number field of degree $2g$ and $\Phi$ 
a set of $g$-distinct $\mbb{Q}$-algebra embeddings from $F$ into $\mbb{C}$.
It is shown in \cite[\S 5.2, Theorem 1]{Sh2} that  
$(F;\Phi)$ is a CM-type if and only if 
$F$ contains a CM field $K$ with the   property that,
for any choice of elements $\vphi\not=\psi$ in $\Phi$, 
it holds $\vphi|_K\not=\bar{\psi}|_K$.
It is also shown in \cite[\S 6.1, Corollary of Theorem 2]{Sh2}
that 
any two abelian varieties of the same CM-type are isogenous to each other.
We say that a CM-type $(F;\Phi)$ is {\it primitive} if 
the abelian varieties of that type are simple.

Let $(F;\Phi)$ be a CM-type.
We set 
$$
F^{\ast}:=\mbb{Q}\left(\sum_{\vphi\in \Phi} \vphi(x)\mid x\in F\right).
$$
If we denote by $F'$  the Galois closure of $F/\mbb{Q}$,
it clearly holds that $F^{\ast}$ is a subfield of $F'$. 
We denote by $\Phi_{F'}$ the set of the elements of $\mrm{Gal}(F'/\mbb{Q})$
inducing some $\vphi\in \Phi$ on $F$ and  
also denote by $\Phi^{\ast}$ the set of all $\mbb{Q}$-algebra embeddings 
of $F^{\ast}$ into $\mbb{C}$ obtained from the  elements of 
$\{\sigma^{-1}\mid \sigma \in \Phi_{F'}\}$.
Then, it is known that $(F^{\ast};\Phi^{\ast})$ is 
a primitive CM-type (cf.\ \cite[\S 8.3, Proposition 28]{Sh2}).
\begin{definition}
We say that $(F^{\ast};\Phi^{\ast})$ is the {\it reflex} of $(F;\Phi)$.
\end{definition}
The reflex satisfies various interesting properties; see \cite[\S 8]{Sh2} for more information.
For our proof of Theorem \ref{Kab:AVKF}, we use the following properties:
\begin{itemize}
\item If $(F;\Phi)$ is a primitive CM-type, then 
$(F;\Phi)$ coincides with the reflex of its reflex.
\item Let $(A,\theta)$ be of CM-type $(F;\Phi)$ 
and $k\subset \mbb{C}$ a subfield.
Assume that $(A,\theta)$ is defined over $k$, that is, 
$A$ is defined over $k$ and 
$\theta(F)\subset \mrm{End}_{k}(A)^0$.
Then we have $k\supset F^{\ast}$.
\end{itemize}

\begin{proof}[Proof of Theorem \ref{Kab:AVKF}]
Consider the following property:
\begin{itemize}
\item[$(3)'$] For every abelain variety $A$ over $K$,
the set of divisible elements of $A(K^{\mrm{ab}})$ is zero.
\end{itemize}
We show implications 
$(1)\Rightarrow (2)\Rightarrow (3)' \Rightarrow (3)\Rightarrow (1)$.

First we show $(1)\Rightarrow (2)$.
Suppose that there exists an abelian variety $A$ over a number field $K$
such that $A(K^{\mrm{ab}})_{\mrm{tor}}$ is infinite.
Choosing $K$-simple abelian varieties $A_1,\dots ,A_r$ over $K$ such that 
$A$ is $K$-isogenous to the product $A_1\times \cdots \times A_r$,
we see that $A_i(K^{\mrm{ab}})_{\mrm{tor}}$ is infinite for some $i$.
Thus we may assume that $A$ is $K$-simple. 
It follows from  \cite[Theorem 1]{Za} that
$F:=\mrm{End}_{K}(A)^0$ is a number field of degree $2\dim A$. 
Let $\theta$ be the natural inclusion  map
$F\hookrightarrow \mrm{End}_{\overline{\mbb{Q}}}(A)^0$ and 
$(F;\Phi)$ the CM-type associated with $(A,\theta)$.
Since every element of $\theta(F)$ is defined over $K$,
it follows from \cite[\S 8.5, Proposition 30]{Sh2} that 
$K$ contains the reflex field of $(F;\Phi)$,
which in particular shows that $K$ contains a CM field.
Thus we obtained $(1)\Rightarrow (2)$.

The implication $(2)\Rightarrow (3)' $ follows from 
\cite[Proposition 2.4]{OT} immediately; 
see also Lemma \ref{AVKF-OT} in the next section. 
(Here, we recall that number fields are AVKF by the Mordell-Weil theorem.)

We show $(3)'\Rightarrow (3)$. 
Assume $(3)'$. Let $L$ be a finite extension of $K^{\mrm{ab}}$ and  
$A$ an abelian variety  over $L$. We want to show that 
the set of divisible elements of $A(L)$ is zero.
If we denote by $B$ the Weil restriction $\mrm{Res}_{L/K^{\mrm{ab}}} (A)$ of $A$,
then  $B$ is an abelian variety over $K^{\mrm{ab}}$ 
and we have $A(L)=B(K^{\mrm{ab}})$.
Take a finite extension $K'$ of $K$ contained in $K^{\mrm{ab}}$
such that $B$ is defined over $K'$, and 
denote by $C$ the Weil restriction $\mrm{Res}_{K'/K} (B)$ of $B$.
Then $C$ is an abelian variety over $K$ and we find
$$
C(K^{\mrm{ab}})=B(K' \otimes_K K^{\mrm{ab}})
=\prod_{\sigma} B^{\sigma} (K^{\mrm{ab}})
\supset B(K^{\mrm{ab}})=A(L),
$$
where
$\sigma$ ranges over the $K$-algebra embeddings of $K'$ into $K^{\mrm{ab}}$
and $B^{\sigma} = B\otimes_{K',\sigma} K^{\mrm{ab}}$
is the base change of $B$ to $K^{\mrm{ab}}$ with respect to $\sigma$.  
By $(3)'$, we conclude that the set of divisible elements of  
$A(L)$ is zero as desired.

Finally we show $(3)\Rightarrow (1)$. 
Assume that $K$ contains a CM field.
The goal is to show that $K^{\mrm{ab}}$ is  not AVKF.
Replacing $K$ by a subfield, 
we may assume that $K$ is a {\it minimal} CM field
in the sense that all the non-trivial subfields of $K$ are not CM fields. 
Take any CM-type $(K;\Phi)$. 
We claim that the CM-type $(K;\Phi)$ is primitive. 
Let us denote by $(A,\theta)$ an abelian variety over $\overline{\mbb{Q}}$ of CM-type $(K;\Phi)$
and denote by $B$ a non-zero simple abelian  $\overline{\mbb{Q}}$-subvariety of $A$. 
Then, $K_0:=\mrm{End}_{\overline{\mbb{Q}}}(B)^0$ is a CM field of degree $2\dim B$ and 
also we may naturally regard $K_0$ as a subfield of $K$ (cf.\ \cite[Chapter 1, Theorem 3.3]{La}).
Since $K$ is a minimal CM field, we have $K_0=K$. 
This gives $\dim A=\dim B$ and thus $A$ is isogenous to $B$,
which implies the fact that $A$ is simple. Thus the claim follows.
Let $(K^{\ast};\Phi^{\ast})$  be the reflex of $(K;\Phi)$
and set $g:=[K^{\ast}:\mbb{Q}]/2$.
We remark that, since  $(K;\Phi)$ is primitive, 
the reflex $(K^{\ast \ast };\Phi^{\ast \ast})$ of $(K^{\ast};\Phi^{\ast})$
coincides with $(K;\Phi)$.
Let $(A,\theta)$ be a $g$-dimensional abelian variety 
over $\overline{\mbb{Q}}$
of type $(K^{\ast};\Phi^{\ast})$.
Since $A$ is simple, the homomorphism 
$\theta\colon K^{\ast}\hookrightarrow \mrm{End}_{\overline{\mbb{Q}}}(A)^0$
is an isomorphism of fields.
Take any polarization $\mcal{C}$ of $A$, defined over $\overline{\mbb{Q}}$,
and consider  the triple $(A,\mcal{C},\theta)$.
As is explained in  \cite[p. 216, after the proof of Theorem 7.44]{Sh1},
there exist a number field $K'$ contained in $(K^{\ast \ast })^{\mrm{ab}}=K^{\mrm{ab}}$
and a triple $(A',\mcal{C}',\theta')$ defined over $K'$
such that  $(A,\mcal{C},\theta)$ is isomorphic to 
$(A',\mcal{C}',\theta')$ over $\overline{\mbb{Q}}$
and all the torsion points of $A'$ are defined over $(K^{\ast \ast })^{\mrm{ab}}=K^{\mrm{ab}}$.
Since any non-zero torsion element of $A'$ is a non-zero divisible element of $A'(K^{\mrm{ab}})$, 
we conclude that $K^{\mrm{ab}}$ is not AVKF as desired.
\end{proof}

\subsection{``Kummer-type"  construction of TKND-AVKF fields}
In this section,
we give a ``Kummer-type"  construction of TKND-AVKF fields (cf.\ Proposition \ref{Thm.A2} and Corollary \ref{sequence}) 
and give a proof of Theorem \ref{MT:AVKF}
in the  Introduction.
It may be helpful to rewrite our results in terms of notions 
appearing in anabelian geometry.
Following \cite[Definition 6.1]{HMT},
we introduce some notions.

\begin{definition}
Let $F$ be a field.

\noindent
(1)  We say that $F$ is {\it AV-tor-finite} if,
for every finite extension $E$ of $F$
and every abelian variety $A$ over $E$,
it holds that $A(E)_{\mrm{tor}}$ is finite. 

\noindent
(2)  Let $p$ be a prime.
We say that $F$ is {\it $p^{\infty}$-AV-tor-finite} if,
for every finite extension $E$ of $F$
and every abelian variety $A$ over $E$,
it holds that $A(E)[p^{\infty}]$ is finite. 
\end{definition}
Number fields are AV-tor-finite by the Mordell-Weil theorem.
More generally, sub-$p$-adic fields are AV-tor-finite
(cf.\ \cite[Proposition 2.9]{OT}).
Ribet's theorem \cite{KL} implies that 
$k^{\mrm{cyc}}$ is AV-tor-finite for any number field $k$.

\begin{remark}
\label{equivRem}
(1) Any subfield of an AV-tor-finite field (resp.\ a $p^{\infty}$-AV-tor-finite field)
is also AV-tor-finite (resp.\ $p^{\infty}$-AV-tor-finite).

\noindent
(2) By considering the Weil restrictions of abelian varieties,
one verifies immediately that 
a field $F$ is AV-tor-finite (resp.\ $p^{\infty}$-AV-tor-finite) 
if and only if $A(F)_{\mrm{tor}}$ (resp.\ $A(F)[p^{\infty}]$) 
is finite for every abelian variety $A$ over $F$. 
\end{remark}

\begin{lemma} %%%%%%%%%%%%%%%%%%%%%%%%%%%%%%%%%%%%%%%%%%%%%%%%%%%%%%%%% Proposition
\label{AVKF-OT}
Let $A$ be an abelian variety over a field $K$ 
and $L$ an algebraic extension of $K$.
Consider the following conditions.
\begin{itemize}
\item[{\rm (a)}] The set of divisible elements of $A(L)$ is zero.
\item[{\rm (b)}] The set of divisible elements of $A(L)_{\rm tor}$ is zero.
\item[{\rm (c)}] $A(L)[\ell^{\infty}]$ is finite for any prime  $\ell$.
\end{itemize}
Then we have $(a)\Rightarrow (b)\Leftrightarrow (c)$. 
If $K$ is AVKF and $L$ is a Galois extension of $K$, then 
we have $(a)\Leftrightarrow (b)\Leftrightarrow (c)$.  
\end{lemma}

\begin{proof}
The same proof as that of \cite[Proposition 2.4]{OT} proceeds.
\end{proof}

Here is  an immediate consequence of  Lemma \ref{AVKF-OT}.
\begin{corollary}
\label{AVKF-OT2}
Let $L$ be a Galois extension of  an AVKF field of characteristic $0$.
Then, $L$ is AVKF if and only if $L$ is 
$p^{\infty}$-AV-tor-finite for every prime $p$.
\end{corollary}

\begin{proposition}
\label{Thm.A2}
Let $K$ be a field of characteristic $0$.
Let $L$ be a Galois extension of $K$ with $L\supset \mu_{\infty}(\overline{K})$ 
and set $M:=L(K^{1/\infty})$.

\noindent
{\rm (1)} Assume that 
$\chi_{p}(G_K)$ is open in $\mbb{Z}_p^{\times}$
for a prime $p$.
If $L$ is $p^{\infty}$-AV-tor-finite, then 
any subfield of $M$ is also $p^{\infty}$-AV-tor-finite.

\noindent
{\rm (2)} Assume that 
$\chi_{p}(G_K)$ is open in $\mbb{Z}_p^{\times}$
for every prime $p$.
If $L$ is AVKF, then any subfield of  $M$ is also AVKF.
\end{proposition}

\begin{proof}
The assertion (2) immediately follows from (1) and Corollary \ref{AVKF-OT2}.
We show (1). 
Let $A$ be an abelian variety over $M$. It suffices to show 
$A(M)[p^{\infty}]$ is finite.
Take any finite extension $K_1$ of $K$ 
contained in $M$ so that $A$ is defined over $K_1$. 
We set $L_1:=LK_1$ and $M_1:=L_1(K_1^{1/\infty})$.
We see that $\chi_p(G_{K_1})$ is open in $\mbb{Z}_p^{\times}$
and $L_1$ is $p^{\infty}$-AV-tor-finite.
Hence, if we denote by $V$ the rational 
$p$-adic Tate module $\mbb{Q}_p\otimes_{\mbb{Z}_p}\plim_n A[p^n]$ 
of $A$, then it holds that  $V^{G_{M_1}}=0$ by Proposition \ref{Thm.A}. 
This implies $A(M_1)[p^{\infty}]$ is finite.
Since $M$ is a subfield of $M_1$, we have done.
\if0
Let $B=\mrm{Res}_{K_1/K}(A)$ be the Weil restriction of $A$.
We have $B(M)=A(K_1\otimes_K M) 
=\prod_{\sigma} A^{\sigma} (M) 
\supset A(M)$, 
where
$\sigma$ ranges over the $K$-algebra embeddings of $K_1$ into $M$
and $A^{\sigma} = A\otimes_{K_1,\sigma} K_1$
is the base change of $A$ to $M$ with respect to $\sigma$.  
Since $L$ is $p^{\infty}$-AV-tor-finite, 
$B(L')[p^{\infty}]$ is finite for any finite extension $L'/L$.
This is equivalent to say that, if we denote by $V$ the rational 
$p$-adic Tate module $\mbb{Q}_p\otimes_{\mbb{Z}_p}\plim_n B[p^n]$ 
of $B$, then it holds that  $V^{G_{L'}}=0$.
By Proposition \ref{Thm.A}, we have $V^{G_M}=0$.
This implies that $B(M)[p^{\infty}]$ is finite, 
which gives the fact that $A(M)[p^{\infty}]$ is finite as desired.
\fi
\end{proof}

By applying Proposition above repeatedly,
we can obtain (TKND-)AVKF criterions
for certain types of fields extensions.

\begin{corollary}
\label{sequence}
Let $K_0\subset K_1 \subset K_2 \subset \cdots \subset K_n=K$
be field extensions in $\overline{\mbb{Q}}$.
Consider the following conditions:
\begin{itemize}
\item[{\rm (i)}] $\chi_p(G_K)$ is open in $\mbb{Z}_p^{\times}$ 
for every prime $p$.
\item[{\rm (i)'}] For any prime $p$,
the absolute ramification index of some finite place of $K$ above $p$
is finite.
\if0
 The absolute ramification index of any finite place of $K$ is finite.
\fi
\item[{\rm (ii)}]  $K_0^{\mrm{cyc}}$ is AVKF.
\item[{\rm (iii)}] $K_i\subset K_{i-1}(K_{i-1}^{1/\infty})$ for any $i$. 
\end{itemize}
Then, we have the followings.

\noindent
{\rm (1)} If   $(i), (ii)$ and $(iii)$ hold,  then
any subfield of $K(K^{1/\infty})$ is AVKF.

\noindent
{\rm (2)} If  $(i)', (ii)$ and $(iii)$ hold,  then
any subfield of $K(K^{1/\infty})$ is TKND-AVKF.

\end{corollary}

\begin{remark}
The assumption (i) in the theorem above holds 
if the absolute ramification index of any finite place of 
the maximal abelian subextension of $K/\mbb{Q}$ is finite.
\end{remark}
\begin{proof}
First we show (1). Note that, for each prime $p$ and 
each $i$,  $\chi_p(G_{K_i})$ is open in $\mbb{Z}_p^{\times}$ by (i).
It follows from (ii) and Proposition \ref{Thm.A2} that $K_0(K_0^{1/\infty})$ is AVKF.
In addition, if $K_{i-1}(K_{i-1}^{1/\infty})$ is AVKF for some $i$, 
then $K_i^{\mrm{cyc}}\, (\subset K_{i-1}(K_{i-1}^{1/\infty}))$ is also AVKF,
which implies that $K_i(K_i^{1/\infty})$ is also AVKF by Proposition \ref{Thm.A2}.
By induction, we obtain the fact that $K_n(K_n^{1/\infty})=K(K^{1/\infty})$ is AVKF.
Next we show (2). It suffices to show that  $K(K^{1/\infty})$ is TKND.
By (i)', we know that $K$ is a generalized sub-$p$-adic field for any prime $p$.
It follows from Lemma D (iii) of \cite{Tsu1} that 
$K$ is both stably $p$-$\times \mu$-indivisible and 
stably $\mu_{p^{\infty}}$-finite for every prime $p$.
By Lemma D (iv) of {\it loc.\ cit.}, $K^{\mrm{cyc}}$ is 
stably $p$-$\times \mu$-indivisible for every prime $p$.
Now we remark that $K(K^{1/\infty})$ is not algebraically closed (since it is AVKF).
Therefore, we conclude that  $K(K^{1/\infty})$ is TKND by Proposition \ref{TKND}.
\end{proof}

\begin{proof}[Proof of Theorem \ref{MT:AVKF}]
First we consider the case (c).
It suffices to show the assertion in the case where  
$K= k(\mu_{p^{n_p}}\mid p\in \mfrak{Primes})$
for some number field $k$, 
and in this case the result follows immediately by applying Corollary  \ref{sequence}
with $K_0=k\subset K_1=K$ 
(note that (ii) in the corollary follows from the result of Ribet \cite{KL}). 

We can reduce a proof of the case (b) to the case (a)
since any group of order $<60$ is solvable.
Thus it suffices to show the theorem in the case (a).
We fix a positive integer $d$ and a number field $k$.
Let $K$ be the composite field of all solvable extensions of degree $\le d$ 
over  $k$. The goal is to show that $K(K^{1/\infty})$ is TKND-AVKF. 
Put $d':=d!$. 
We denote by $K'$ the composite of all finite extension fields $k'$ over $k$
with the following properties:
\begin{itemize}
\item[(i)] $[k':k]\le d'$, and 
\item[(ii)] For some $m\ge 0$, $k'/k$ admits a finite subextensions 
$k=k_0'\subset k_1'\subset \cdots \subset k_m'=k'$
such that each $k_i'/k_{i-1}'$ is abelian.
(We say that $k'/k$ is {\it of length $\le m$} if $k'$ satisfies this situation.)
\end{itemize} 
We have $K\subset K'$ since the degree of the Galois closure of a degree $d$ field extension is
at most $d'$.
Hence, it suffices to show that $K'(K'^{1/\infty})$ is TKND-AVKF. 
Replacing $k$ by a finite extension, 
we may assume that $k$ contains all $d'!$-th roots of unity.
We denote by $K'_i$ the composite of all finite extension fields $k'$ over $k$ 
such that both $[k':k]\le d'$ and $k'$ is of length $\le i$.
By definition, we have the following field extensions;
$$
k=K'_0\subset K'_1\subset K'_2\subset \cdots \subset K'_{d'}=K'.
$$
By a similar manner as Example \ref{st-st},
we see that  the absolute ramification index of any finite place of $K'$ is finite.
The field $(K'_0)^{\mrm{cyc}}=k^{\mrm{cyc}}$ is AVKF by the theorem of Ribet \cite{KL}.
Furthermore, it follows from Kummer theory that 
$K'_i\subset K'_{i-1}(K'_{i-1}{}^{1/\infty})$ for each $i$
(here we note that $k$ contains all $d'!$-th roots of unity).
Therefore, we conclude that $K'(K'^{1/\infty})$ is TKND-AVKF
by Corollary \ref{sequence}. 
\end{proof}

Let $v_p$ be the $p$-adic valuation normalized by $v_p(p)=1$.

\begin{corollary}
\label{sequence2}
Let $K_0\subset K_1 \subset K_2 \subset \cdots \subset \bigcup^{\infty}_{n=1} K_n =K$
be field extensions in $\overline{\mbb{Q}}$.
Assume that the following conditions hold:
\begin{itemize}
\item[{\rm (i)}] $K_0$ is a number field and 
$K_i/K_{i-1}$ is an abelian extension with finite exponent for each $i$. 
\item[{\rm (ii)}]  For any prime $p$, the extension $K_i/K_{i-1}$ is prime-to-$p$ 
for any $i$ large enough.
\end{itemize}
Then, $K(K^{1/\infty})$ is TKND-AVKF.
\end{corollary}

\begin{remark}
We cannot remove the assumption (ii) from the statement of Corollary \ref{sequence2}.
Let $E$ be an elliptic curve defined over a number field $k$ 
such that  $\mrm{End}_{k}(E)\otimes_{\mbb{Z}}\mbb{Q}$ is a CM field
and denote by $K_i=k(E[p^i])$  the extension field of $k$ obtained by adjoining 
all torsion points of $E$ killed by $p^i$.
One sees that the assumption (i) is satisfied for $\{K_i\}_i$ but (ii) is not satisfied.
The field $K:=\bigcup_i K_i$ is TKND (by Proposition \ref{TKND} and 
the fact that $k^{\mrm{cyc}}\, (\subset K)$ 
is stably $p$-$\times \mu$-indivisible for any prime $p$)  
but is not AVKF.
\end{remark}

\begin{proof}[Proof of Corollary \ref{sequence2}]
In this proof, we set $\mu_{n}:=\mu_{n}(\overline{\mbb{Q}})$
for any integer $n>0$ to simplify notation.
By assumptions, there exists a family $\{c_i\}_{i\ge 1}$ 
of non-negative integers $c_i$ such that 
$\mrm{Gal}(K_i/K_{i-1})$ is of exponent $c_i$ for each $i$
and, for any prime $p$, $p$ does not divide $c_i$ for any $i$ large enough.
For the proof, we may assume that each $K_i$ is a Galois 
extension of $K_0$.
(Indeed, if we denote by $\widetilde{K}_i/K_0$ the Galois closure of $K_i/K_0$,
then 
$\widetilde{K}_i$ is the composite of all $\sigma(K_i)$ for all $\sigma\in G_{K_0}$
and thus we have  injections 
$\mrm{Gal}(\widetilde{K}_i/\widetilde{K}_{i-1})
\hookrightarrow \prod_{\sigma\in G_{K_0}} 
\mrm{Gal}(\sigma(K_i)\widetilde{K}_{i-1}/\widetilde{K}_{i-1})
\hookrightarrow \prod_{\sigma\in G_{K_0}} 
\mrm{Gal}(\sigma(K_i)/\sigma(K_{i-1}))$,
which implies the fact that  $\widetilde{K}_i/\widetilde{K}_{i-1}$ is an abelian extension with  
exponent $c_i$ for each $i$.)
In particular, $K$ is a Galois extension of $K_0$.

\vspace{2mm}

{\bf STEP 1.} We show that 
$\chi_p(G_K)$ is open in $\mbb{Z}_p^{\times}$ for any prime $p$. 
Assume that $\chi_p(G_K)$ is not open.
Set $p':=p$ or $p':=4$ if $p\not=2$ or $p=2$,
respectively.
Replacing $K_i(\mu_{p'})$ with $K_i$,
we may furthermore assume that $K_0$ contains $\mu_{p'}$.
Since $K$ contains $\mu_{p'}$, the field $K$ contains all $p$-power roots of unity.
Take an integer $n$ large enough such that 
\begin{equation}
\label{pval}
v_p([K_0(\mu_{p^{n}}):K_0])>\sum^{\infty}_{j=1} v_p(c_j).
\end{equation}
(Note that the right hand side of \eqref{pval} is finite by the assumption (ii).)
Since $K_0$ contains $\mu_{p'}$, 
the Galois group $\mrm{Gal}(K_0(\mu_{p^n})/K_0)$ is cyclic of $p$-power order;
we put $[K_0(\mu_{p^{n}}):K_0]=p^m$.
We also take an integer $i$ large enough such that $K_i$ contains  $\mu_{p^{n}}$.
Take an element 
$\sigma \in \mrm{Gal}(K/K_0)$ such that its restriction $\bar{\sigma}$
to $K_0(\mu_{p^n})$ generates $\mrm{Gal}(K_0(\mu_{p^n})/K_0)$.
By the assumption (i), we find that $\sigma^{c_1c_2\cdots c_i}$
fixes $K_i$ and thus $\bar{\sigma}^{c_1c_2\cdots c_i}$ is trivial.
Since the order of $\bar{\sigma}$ is $p^m$, we see that $p^m$
divides $c_1c_2\cdots c_i$ but this contradicts \eqref{pval}.

\vspace{2mm}

{\bf STEP 2.} We show that $K_i(K_i^{1/\infty})$ is AVKF for any $i\ge 1$
(thus $K_i^{\mrm{cyc}}$ is also AVKF).
Put $c=c_1c_2\cdots c_i$ 
and $\hat K_j=K_j(\mu_c)$ for $1\le j \le i$.
The field extensions 
$\hat K_0\subset \hat K_1 \subset \hat K_2 \subset \cdots \subset \hat K_i$
satisfy the property that 
$\hat K_j/\hat K_{j-1}$ is an abelian extension with exponent $c_j$ for each $j$. 
Since $\hat K_{j-1}$ contains $\mu_{c_j}$,
it follows from Kummer theory that $\hat K_j\subset \hat K_{j-1}(\hat K_{j-1}^{1/c_j})$
for each $j$.
Furthermore, $\chi_p(G_{\hat K_i})$ is open in $\mbb{Z}_p^{\times}$ by STEP 1
and $\hat K_0^{\mrm{cyc}}$ is AVKF by the theorem of Ribet \cite{KL}.
Hence we obtain that $\hat K_i(\hat K_i^{1/\infty})$ is AVKF
by  Corollary \ref{sequence}.
This in particular implies  that $K_i(K_i^{1/\infty})$ is also AVKF as desired.

\vspace{2mm}

{\bf STEP 3.} We show that $K(K^{1/\infty})$ is AVKF.
It suffices to show that $K^{\mrm{cyc}}$ is AVKF by STEP 1
and Proposition \ref{Thm.A2} (2).
Here we recall that $K$ is now a Galois extension of $K_0$.
Hence $K^{\mrm{cyc}}$ is a Galois extension of 
an AVKF field $K_0$.
By Remark \ref{equivRem} (2) and Corollary \ref{AVKF-OT2}, 
it is enough to prove that  
$A(K^{\mrm{cyc}})[p^{\infty}]$ is finite for any 
abelian variety $A$ over $K^{\mrm{cyc}}$ and any prime $p$.
Now we assume that $A(K^{\mrm{cyc}})[p^{\infty}]$ is infinite
for some prime $p$.
Then the $G_{K^{\mrm{cyc}}}$-fixed part $V^{G_{K^{\mrm{cyc}}}}$ of 
the rational $p$-adic Tate module   $V:=(\plim_{n} A[p^{n}])\otimes_{\mbb{Z}_p} \mbb{Q}_p$
is not zero.
Let $g$ be the dimension of $A$ and choose a finite subextension $K_0'$ of $K^{\mrm{cyc}}/K_0$
such that $A$ is defined over $K_0'$.
Then the $\mbb{Q}_p$-dimension of $V^{G_{K^{\mrm{cyc}}}}$ is at most $2g$ 
and the Galois group $\mrm{Gal}(K^{\mrm{cyc}}/K_0')$
acts continuously on $V^{G_{K^{\mrm{cyc}}}}$.
By continuity of a Galois action, there exists a $\mbb{Z}_p$-lattice $\mcal{L}$ in 
 $V^{G_{K^{\mrm{cyc}}}}$ which is stable under the $\mrm{Gal}(K^{\mrm{cyc}}/K_0')$-action. 
The $\mrm{Gal}(K^{\mrm{cyc}}/K_0')$-action on $\mcal{L}$ is given by a continuous 
homomorphism $\rho\colon \mrm{Gal}(K^{\mrm{cyc}}/K_0')\to GL_{\mbb{Z}_p}(\mcal{L})
\simeq GL_{t}(\mbb{Z}_p)$
for some $t\le 2g$.
Take an integer $i$ large enough such that 
$K_i$ contains $K_0'$, and that $\mrm{Gal}(K^{\mrm{cyc}}/K_i^{\mrm{cyc}})$  
is pro-prime to
the order of $GL_{t}(\mbb{Z}/p'\mbb{Z})$ and $p$
(such $i$ exists by the assumption (ii)). 
Then the restriction to $\mrm{Gal}(K^{\mrm{cyc}}/K_i^{\mrm{cyc}})$ of 
the composite of $\rho$ and the projection 
$GL_{t}(\mbb{Z}_p)\to GL_{t}(\mbb{Z}/p'\mbb{Z})$
has trivial image.
Thus $\rho$ restricted to $\mrm{Gal}(K^{\mrm{cyc}}/K_i^{\mrm{cyc}})$ 
has values in the kernel of the projection 
$GL_{t}(\mbb{Z}_p)\to GL_{t}(\mbb{Z}/p'\mbb{Z})$.
Since this kernel is pro-$p$,
we obtain the fact that 
$\rho$ restricted to $\mrm{Gal}(K^{\mrm{cyc}}/K_i^{\mrm{cyc}})$ 
must be trivial.
This implies that $V^{G_{K^{\mrm{cyc}}}}=V^{G_{K_i^{\mrm{cyc}}}}$. 
Thus $V^{G_{K_i^{\mrm{cyc}}}}$ is not zero by the assumption that $A(K^{\mrm{cyc}})[p^{\infty}]$ is infinite
but this contradicts the fact proved in STEP 2 that $K_i^{\mrm{cyc}}$ is  AVKF.
 
\vspace{2mm}

{\bf STEP 4.} We end the proof by proving that $K(K^{1/\infty})$ is TKND. 
Similar to the proof of Corollary \ref{sequence} (2),
we follow  Tsujimura's results \cite{Tsu1}.
By STEP 1, we know that $K$ is stably $\mu_{p^{\infty}}$-finite.
Furthermore, since $K$ is a Galois extension of $K_0$ and $K_0$
is stably $p$-$\times \mu$-indivisible, 
it follows from Lemma D (v) of \cite{Tsu1} that 
$K$ is stably $p$-$\times \mu$-indivisible.
By Lemma D (iv) of {\it loc. cit.} (or Example \ref{st-st}), 
we find that 
$K^{\mrm{cyc}}$ is also stably $p$-$\times \mu$-indivisible.
On the other hand, we know that $K(K^{1/\infty})$ is not algebraically closed since
it is AVKF.
By Proposition \ref{TKND},
we conclude that $K(K^{1/\infty})$ is TKND.
\end{proof}

\begin{example}
Let $k$ be a number field and $\{c_p\}_{p\in \mfrak{Primes}}$   a family of non-negative integers $c_p$.
Let $K$ be the composite of all finite abelian extensions $k'$ of $k$ with the property that 
$v_p([k':k])\le c_p$ for any prime $p$.
Then $K(K^{1/\infty})$ is TKND-AVKF. 

This can be checked as follows:
Let $p_i$ be the $i$-th prime, that is, $p_1=2<p_2=3<p_3=5<\cdots$.
Consider the field extensions 
$K_0\subset K_1 \subset K_2 \subset \cdots \subset \bigcup^{\infty}_{n=1} K_n =K_{\infty}$ 
defined inductively so that $K_0=k$ and $K_{i}$ is the composite of all abelian extensions over $K_{i-1}$ 
of degree dividing $p_i^{c_{p_i}}$.
Corollary \ref{sequence2} asserts that  $K_{\infty}(K_{\infty}^{1/\infty})$ is TKND-AVKF. 
Since $K$ is a subfield of $K_{\infty}$, the result follows.
\end{example}

We end this paper with an AV-tor-finite analogue of 
Proposition \ref{Thm.A2}.

\begin{proposition}
\label{Thm.B2}
Let $K$ be a field of characteristic $0$ such that both
$\chi_{p}(G_K)$ is open in $\mbb{Z}_{p}^{\times}$ for every prime $p$
and, for some integer $\mu>0$, 
$\bar{\chi}_{p}(G_K)\supset (\mbb{F}_{p}^{\times})^{\mu}$ 
for all but finitely many primes $p$.
Let $L$ be a Galois extension of $K$ with $L\supset \mu_{\infty}(\overline{K})$.
Let $\{\Delta_{p}\}_{p\in \mfrak{Primes}}$ be a family of 
finitely generated subgroups $\Delta_{p}$ of $K^{\times}$.
We set 
$$
M:=L(\Delta_{p}^{1/p^{\infty}}\mid p\in \mfrak{Primes}).
$$
If $L$ is AV-tor-finite, then any subfield of  $M$ is also AV-tor-finite.
\end{proposition}

\begin{proof}
Let $A$ be an abelian variety over $M$. 
It follows from Proposition \ref{Thm.A2} that $A(M)[\ell^{\infty}]$ is finite
for every prime $\ell$.
(Note that AV-tor-finite fields are $\ell^{\infty}$-AV-tor-finite for every prime $\ell$.)
Thus it suffices to show that 
$A(M)[\ell]=0$ for all but finitely many primes $\ell$.
Take any finite extension $K_1$ of $K$ 
contained in $M$ so that $A$ is defined over $K_1$. 
We set $L_1:=LK_1$ and $M_1:=L_1M=L_1(\Delta_{p}^{1/p^{\infty}}\mid p\in \mfrak{Primes})$.
We see that $\chi_p(G_{K_1})$ is open in $\mbb{Z}_p^{\times}$
and $L_1$ is AV-tor-finite.
Furthermore, putting $d=[K_1:K]$, 
it follows that 
$\bar{\chi}_{p}(G_{K_1})\supset (\mbb{F}_{p}^{\times})^{d\mu}$ 
for all but finitely many primes $p$.
(In fact, we have $\bar{\chi}_{p}(G_{K})^d\subset \bar{\chi}_{p}(G_{K_1})$ 
since the index of $\bar{\chi}_{p}(G_K)/\bar{\chi}_{p}(G_{K_1})$ is $[K(\mu_p)\cap K_1:K]$
which is a divisor of $d$.)
On the other hand, putting $g=\dim A$, 
we know that $W_{\ell}:=A[\ell]$ is a $2 g$-dimensional 
$\mbb{F}_{\ell}$-representation of $G_{K_1}$.
Applying Proposition \ref{Thm.B}, we find that  $W_{\ell}^{G_{M_1}}=0$, 
equivalently $A(M_1)[\ell]=0$,  for all but finitely many primes $\ell$. 
Since $M$ is a subfield of $M_1$, we finish a proof.
\end{proof}

\end{document}